\documentclass{article}
\usepackage{amssymb}
\usepackage{amsfonts}
\usepackage{amsmath}
\usepackage{amsmath}
\usepackage{amsfonts}
\usepackage{amssymb}
\usepackage{setspace}
\usepackage{enumitem}
\usepackage{chngcntr}
\usepackage{apptools}
\usepackage[bottom]{footmisc}
\usepackage[round]{natbib}   
\usepackage{color}

\newtheorem{theorem}{Theorem}
\newtheorem{assumption}[theorem]{Assumption}

\newtheorem{lemma}[theorem]{Lemma}

\newtheorem{remark}[theorem]{Remark}

\newenvironment{proof}[1][Proof]{\noindent \textbf{#1.} }{\  \rule{0.5em}{0.5em}}
\input{tcilatex}

\title{\vspace{-4cm}Exploration versus exploitation in reinforcement learning: a stochastic control approach\thanks{We are grateful for comments from the seminar participants at  UC Berkeley and Stanford, and from the participants at the Columbia Engineering for Humanity Research Forum
``Business Analytics; Financial Services and Technologies" in New York and
The Quantitative Methods in Finance 2018 Conference in Sydney. We thank Jose Blanchet, Wendell Fleming, Kay Giesecke, Xin Guo, Josef Teichmann
and Renyuan Xu for helpful discussions and comments on the paper.}}
\author{Haoran Wang\footnote{Department of Industrial Engineering and Operations Research, Columbia University, New York, NY 10027, USA. Email: hw2718@columbia.edu.}\and Thaleia Zariphopoulou\footnote{Department of Mathematics and IROM, The University of Texas at Austin, Austin, USA and the Oxford-Man Institute, University of Oxford, Oxford, UK. Email: zariphop@math.utexas.edu.}\and Xun Yu Zhou\footnote{Department of Industrial Engineering and Operations Research, and The Data Science Institute, Columbia University, New York, NY 10027, USA. Email: xz2574@columbia.edu.}}
\date{\vspace{-5ex}}

\begin{document}
\maketitle
\begin{center}
First draft: March 2018\\
This draft: February 2019
\end{center}

\bigskip
\bigskip

\begin{abstract}
We consider reinforcement learning (RL) in continuous time and study the problem of
achieving the best trade-off between exploration of a black box environment and exploitation of current knowledge. We propose an entropy-regularized  reward function
involving the differential entropy of the distributions of actions, and motivate and devise an exploratory formulation for
the feature dynamics that captures repetitive learning under exploration.
The resulting optimization problem is a revitalization  of the classical relaxed stochastic control. We carry out a complete analysis of the problem in the linear--quadratic (LQ) setting and deduce that the optimal feedback control distribution
for balancing exploitation and exploration  is Gaussian. This in turn interprets and justifies the widely adopted Gaussian exploration in RL, beyond its simplicity for sampling. Moreover, the exploitation and exploration are captured, respectively and mutual-exclusively, by the mean and variance of the Gaussian distribution.
We also find that a more random environment contains more learning opportunities in the sense that less exploration is needed.
We characterize the cost of exploration, which, for the LQ case, is shown to be proportional to the entropy regularization weight and inversely proportional to the discount rate. Finally, as the weight of exploration decays to zero, we prove the convergence of the solution of the entropy-regularized  LQ problem to the one of the classical LQ problem.

\end{abstract}
\newpage

\newpage

\section{Introduction}
Reinforcement learning (RL) is currently  one of the most active and fast developing subareas in machine learning. In recent years, it has been successfully applied to solve large scale real world, complex decision making problems, including playing perfect-information board games such as Go (AlphaGo/AlphaGo Zero, \cite{Go1}, \cite{Go2}), achieving human-level performance in video games (\cite{DQN}), and driving autonomously (\cite{robot1}, \cite{robot2}). An RL agent does not pre-specify a structural model or a family of models but, instead, gradually learns the best (or near-best) strategies  based on trial and error, through interactions with the random (black box) environment and incorporation of the responses of these interactions, in order to improve the overall performance. This is a case of ``kill two birds with one stone": \textit{the agent's actions (controls) serve both as a means to explore (learn) and a way to exploit (optimize)}.

Since exploration is inherently costly in terms of resource, time and opportunity, a natural and crucial question in RL is to address the dichotomy between exploration of uncharted territory and exploitation of
existing knowledge. Such question exists in both the stateless RL settings (e.g. the multi-armed bandit problem) and the more general multi-state RL settings (e.g. \cite{SB}, \cite{Ka}). Specifically,  the agent must balance between greedily exploiting what has been learned so far to choose actions that yield near-term higher rewards, and continuously exploring the environment to acquire more information to potentially achieve long-term benefits. Extensive studies have been conducted to find strategies for the best trade-off betweeen exploitation and exploration.\footnote{For the multi-armed bandit problem, well known strategies include Gittins index approach (\cite{Gittins}), Thompson sampling (\cite{Thompson}), and upper confidence bound algorithm (\cite{UCB}), whereas theoretical optimality is established, for example, in \cite{RV2,RV1}. For general RL problems, various efficient exploration methods have been proposed that have been proved to induce low sample complexity, among other advantages (see, for example, \cite{RMax}, \cite{MBIE}, \cite{DelayedQ}).}

However, most of the contributions to balancing exploitation and exploration do not include exploration explicitly  as a part of the optimization objective;  the attention has mainly focused on solving the classical optimization problem maximizing the accumulated rewards, while exploration is typically treated separately as an {\it ad-hoc} chosen exogenous component, rather than being endogenously {\it derived} as a part of the solution to the overall RL problem. The recently proposed discrete time entropy-regularized  (also termed as
``entropy-augmented" or ``softmax")  RL formulation, on the other hand, explicitly incorporates  exploration into the optimization objective as a regularization term, with a trade-off weight imposed on the entropy of the exploration strategy (\cite{Z}, \cite{Na}, \cite{Glearning}). An exploratory distribution with a greater entropy signifies  a higher level of exploration, reflecting a bigger weight on the exploration front. On the other hand, having the minimal entropy, the extreme case of Dirac measure implies no exploration, reducing to  the  case of
 classical optimization with a complete knowledge about the underlying model.
 Recent works have been devoted to the designing of various algorithms to solve the entropy regularized RL problem, where numerical experiments have demonstrated remarkable robustness and multi-modal policy learning (\cite{Ha}, \cite{Ha2}).

 In this paper, we study the trade-off between exploration and exploitation for RL
 in a continuous-time setting with both continuous control (action) and
state (feature) spaces.\footnote{The terms ``feature" and ``action" are
typically used in the RL literature, whose counterparts in the control literature
are ``state" and ``control", respectively. Since this paper uses the control approach to study RL problems, we will interchangeably use these terms
whenever there is no confusion.}
 Such a continuous-time formulation is especially appealing if the agent can interact with the environment at ultra-high frequency, examples including
 high frequency stock trading, autonomous driving and snowboard riding.
  More importantly, once cast in continuous time, it is possible, thanks in no small measure to the tools of stochastic calculus and differential equations, to derive  elegant and insightful results which, in turn, lead to theoretical  understanding of some of the fundamental issues in RL, give guidance to algorithm design and provide {\it interpretability} to the underlying learning technologies.

Our first main contribution is to  propose an {\it entropy-regularized  reward function} involving the differential entropy for exploratory probability distributions over the continuous action space, and motivate and devise an ``exploratory formulation"  for the state dynamics that captures  repetitive learning under exploration in the continuous time limit. Existing theoretical works on exploration mainly concentrate on the analysis at the algorithmic level, including proving convergence of the proposed exploration algorithms to the solutions of the classical optimization problems  (see, for example, \cite{greedy}, \cite{convergence}). However, they rarely look into the  impact of the exploration on
changing significantly the underlying dynamics (e.g. the transition probabilities
in the discrete time context).
Indeed, exploration not only substantially enriches the space of control strategies (from that of Dirac measures to that of all possible probability distributions) but also, as a result, enormously expands the reachable space of states.
This, in turn, sets out to change both the underlying state transitions and the system dynamics.

We show that our exploratory  formulation can account for the effects of learning in both the rewards received and the state transitions observed from the interactions with the environment. It, thus, unearths the important characteristics of learning at a more refined and in-depth level, beyond merely devising and analyzing  learning algorithms. Intriguingly, the proposed formulation of the state dynamics coincides with that in
the {\it relaxed control} framework in classical control theory (see, for example, \cite{Nisio, Karouni, Zhou, Kurtz1, Kurtz2}), which was motivated by entirely different reasons. Specifically, relaxed controls were introduced to mainly deal with the {\it theoretical} question of whether an optimal control exists. The approach essentially entails randomization to convexify the universe of control strategies.
To the best of our knowledge, the present paper is the first to bring back the formulation of
relaxed control, guided by a practical motivation: exploration and learning.

We then carry out a complete analysis on the continuous-time entropy-regularized
RL problem, assuming that the original system dynamics is linear in both the control and the state, and that the original reward function is quadratic in the two. This type of linear--quadratic (LQ) problems has occupied the center stage for research in classical control theory for its elegant solutions and its ability to approximate more general nonlinear problems.
One of the most important, conceptual contributions of this paper is to show that the optimal feedback control distribution for balancing exploitation and exploration is {\it Gaussian}. Precisely speaking, if, at any given state, the agent sets out to engage in exploration then she needs look no further than Gaussian distributions.
As is well  known, a pure exploitation optimal distribution is Dirac, and a pure exploration
optimal distribution is uniform. Our results reveal that Gaussian is the right
choice if one seeks a balance between those two extremes.
Moreover, we find that the mean of this optimal exploratory distribution is a function of the current state {\it independent} of the intended exploration level, whereas the variance is a linear function of the entropy regularizing  weight (also called the
``temperature parameter" or ``exploration weight") {\it irrespective} of the
current state. This result highlights a {\it separation} between exploitation and exploration: the former is reflected in the mean and the latter in the variance of
the optimal Gaussian distribution.

There is yet another intriguing result. The higher impact actions have on the  volatility of  the underlying dynamic system, the smaller the variance of the optimal Gaussian distribution needs to be. Conceptually, this implies that a more random environment  in fact contains more learning opportunities and,
hence, is less costly for learning. This theoretical finding provides an interpretation of the recent RL heuristics where injecting
noises leads to better effect of exploration; see, for example, \cite{Lillicrap,Plappert}.

Another contribution of the paper is that we establish a direct connection between the solvability of the exploratory LQ problem and that of the classical LQ problem. We prove that as the exploration weight in the former decays to zero, the optimal Gaussian control distribution and its value function converge respectively to the optimal Dirac measure and the value function of the classical LQ problem, a desirable result for practical learning purposes.

Finally, we observe that, beyond the LQ problems and under proper conditions, the Gaussian distribution remains optimal for a much larger class of control problems, namely, problems with drift and volatility linear in control and reward functions linear or quadratic in control even if the dependence on state is nonlinear. 
Such a family of problems can be seen as the local-linear-quadratic approximation to more general stochastic control problems whose state dynamics are linearized in the control variables  and the reward functions are locally approximated by quadratic control functions (\cite{LT1}, \cite{LT2}). Note also that although such iterative LQ approximation generally has different parameters at different local state-action pairs, our result on the optimality of Gaussian distribution under the exploratory  LQ framework still holds at any local point, and therefore justifies, from a stochastic control perspective, why Gaussian distribution is commonly used in the RL practice for exploration (see, among others, \cite{Ha}, \cite{Ha2}, \cite{Na2}), beyond its simplicity for sampling.

The rest of the paper is organized as follows. In section 2, we motivate
and propose the relaxed stochastic control formulation involving an exploratory
state dynamics and an entropy-regularized reward function for our RL problem.  We then present the associated Hamilton-Jacobi-Bellman (HJB) equation and the optimal control distribution for general entropy-regularized  stochastic control problems in section 3. In section 4, we study the special  LQ problem in both the state-independent and state-dependent reward cases, corresponding respectively to the multi-armed bandit problem and the general RL problem in discrete time, and derive the optimality of Gaussian exploration.  We discuss the connections between the exploratory LQ problem and the classical LQ problem in section 5, establish the solvability equivalence of the two and the convergence result for vanishing exploration, and finally characterize the cost of exploration. We conclude in section 6. Some technical contents and proofs are relegated to Appendices.

\section{An Entropy-Regularized Relaxed Stochastic Control Problem}

We introduce an entropy-regularized relaxed stochastic control problem and provide
its motivation in the context of RL.

Consider a filtered probability space $\left( \Omega ,\mathcal{F},\mathbb{%
P}; \{\mathcal{F}_t\}_{t\geq0}\right) $ in which we define an $\{\mathcal{F}_t\}_{t\geq0}$-adapted Brownian motion $W=\{W_{t},$ $t\geq 0\}.$ An ``action space" $U$ is given, representing the constraint on
an agent's decisions (``controls" or ``actions"). An admissible ({\it open-loop}) control $u=\{u_{t},$ $t\geq 0\}$ is an
$\{\mathcal{F}_t\}_{t\geq0}$-adapted measurable process taking values in $U$.

The classical stochastic control problem is to control the state (or ``feature") dynamics\footnote{We assume that both the state and the control are scalar-valued, only for notational simplicity. There is no essential difficulty to carry out our discussions with these being vector-valued.}
%
%
%
\begin{equation}
dx_{t}^{u}=b(x_{t}^{u},u_{t})dt+\sigma (x_{t}^{u},u_{t})dW_{t},\; t>0;\quad  x_{0}^{u}=x\in \mathbb{R},
\label{classical_state}
\end{equation}%
where (and throughout this paper) $x$ is a generic variable representing a current state of the system dynamics. The aim of the control is to achieve the maximum expected total discounted reward represented by the  value function
\begin{equation}
V^{\text{cl}}\left( x\right) :=\sup_{u\in\mathcal{A}^{\text{cl}}(x)}\mathbb{E}\left[ \left.
\int_{0}^{\infty }e^{-\rho t}r\left( x_{t}^{u},u_{t}\right) dt\right \vert
x_{0}^{u}=x\right],  \label{classical}
\end{equation}%
where $r$ is the reward function, $\rho>0$ is the discount rate, and $\mathcal{A}^{\text{cl}}(x)$ denotes the set of all admissible controls which in general may depend on $x$.

In the classical setting, where the model is fully known (namely, when the functions $b,\sigma$ and $r$ are fully specified) and the dynamic programming is applicable, the optimal control can be derived and represented as a {\it deterministic} mapping from the current state to the action space $U$, $u^*_t=\boldsymbol{u}^*{(x^*_t)}$.
The mapping $\boldsymbol{u}^*$ is called an optimal {\it feedback} control (or ``policy" or ``law"); this feedback control is derived at $t=0$ and {\it will} be carried out
through $[0,\infty)$.\footnote{In general, feedback controls are easier to implement as they respond directly to the {\it current} states of the controlled dynamics.}


In contrast, in the RL setting, where the underlying model is not known and therefore dynamic learning is needed, the agent employs exploration to interact with and learn the
unknown environment through trial and error. The key idea is to model exploration by
a {\it distribution} of controls $\pi=\{\pi_t(u),t\geq0\}$ over the control space $U$ from which each ``trial" is sampled.\footnote{As will be evident below, rigorously speaking, $\pi_t(\cdot)$ is a probability {\it density} function for each $t\geq0$. With a slight abuse of terminology, {\it we will not distinguish a density function from its corresponding probability distribution or probability measure and thus will use these terms interchangeably in this paper}. Such nomenclature is common in the RL literature.}
We can therefore extend the notion of controls to distributions.\footnote{A classical control $u=\{u_{t},$ $t\geq 0\}$ can be regarded as a Dirac distribution (or ``measure") $\pi=\{\pi_t(u),t\geq 0\}$ where $\pi_t(\cdot)=\delta_{u_t}(\cdot)$. In a similar fashion, a feedback policy
$u_t=\boldsymbol{u}{(x^u_t)}$ can be embedded as
 a Dirac measure $\pi_t(\cdot)=\delta_{\boldsymbol{u}{(x^u_t)}}(\cdot)$, parameterized by the current state ${x^u_t}$.}
The agent executes a  
control  for $N$ rounds over the same time
horizon, while at each round, a classical control is sampled from the distribution ${\pi}$. The reward of such a policy becomes accurate
enough when $N$ is large. This procedure, known as \textit{policy evaluation}, is considered  as
a fundamental element of most RL algorithms in practice
(\cite{SB}). Hence, for evaluating such a policy distribution in our continuous time setting, it is necessary to consider the limiting situation  as $N\rightarrow \infty $.

In order to capture the essential idea for doing this, let us first examine the special case when the reward
depends only on the control, namely, $r({x^u_t},u_{t})=$ $r(u_{t}).$ One then
considers $N$ identical independent copies of the control problem in the following way: at round $i$, $i=1,2,\dots ,N,$ a control $u^{i}$ is sampled under the
(possibly random) control distribution $\pi$, and executed for its corresponding copy of the control
problem (\ref{classical_state})--(\ref{classical}).
Then,
at each fixed time $t$, it follows, from the law of large numbers
(and under certain mild technical conditions), that the average reward over $%
[t,t+\Delta t]$, with $\Delta t$ small enough, should satisfy that as $N\rightarrow \infty$,
$$\frac{1}{N}\sum_{i=1}^Ne^{-\rho t}r(u_t^i)\Delta t\xrightarrow{\ \text{a.s.}\ } \mathbb{E}\left[e^{-\rho t}\int_{U}r(u)\pi_t(u)du\Delta t\right].$$

%

For a general reward  $r({x^u_t},u_{t})$ which depends on the state, we first need to describe how exploration
might alter the state dynamics (\ref{classical_state}) by
defining appropriately its
``exploratory" version.
For this, we look at the effect of repetitive learning under a given control distribution, say $\pi $, for $N$ rounds.
Let $W_{t}^{i}$, $i=1,2,\dots ,N$, be $N$
independent sample paths of the Brownian motion $W_{t}$, and $%
x_{t}^{i}$, $i=1,2,\dots ,N$,  be the
 copies of the state process respectively under the controls $u^{i}$, $i=1,2,\dots ,N$, each  sampled from $\pi $.
Then, the increments of these state process copies are, for $i=1,2,\dots ,N$,
\begin{equation}
\Delta x_{t}^{i}\equiv x_{t+\Delta t}^{i}-x_{t}^{i}\  \approx \ b(x_{t}^{i},u_{t}^{i})\Delta t+\sigma
(x_{t}^{i},u_{t}^{i})\left( W_{t+\Delta t}^{i}-W_{t}^{i}\right) ,\quad t\geq
0.  \label{independent_copy}
\end{equation}%
Each such process ${x^{i}}$, $i=1,2,\dots ,N$, can be viewed as an independent
sample from the exploratory state dynamics $X^{\pi }$.
The superscript $\pi $ of $X^{\pi }$ indicates that each ${x^{i}}$ is
generated according to the classical dynamics (\ref{independent_copy}), with
the corresponding $u^{i}$ sampled independently under this policy $%
\pi.$
%


It then follows from  (\ref{independent_copy}) 
and the law of large numbers that, as $N\rightarrow \infty $,
\begin{equation}
\begin{array}{rl}
&\hspace{-2pt}\frac{1}{N}\sum_{i=1}^{N}\Delta x_{t}^{i}\  \approx \  \frac{1}{N}%
\sum_{i=1}^{N}b(x_{t}^{i},u_{t}^{i})\Delta t+\frac{1}{N}\sum_{i=1}^{N}\sigma
(x_{t}^{i},u_{t}^{i})\left( W_{t+\Delta t}^{i}-W_{t}^{i}\right)\\
& \\
\xrightarrow{\ \text{a.s.}\ } & \mathbb{E}\left[ \int_{U}b(X_{t}^{\pi },u)\pi
_{t}(u)du\Delta t\right] +\mathbb{E}\left[ \int_{U}\sigma (X_{t}^{\pi
},u)\pi _{t}(u)du\right] \mathbb{E}\left[ W_{t+\Delta t}-W_{t}\right]\\ \label{drift_convergence}
& \\
=&\mathbb{E}\left[ \int_{U}b(X_{t}^{\pi },u)\pi_{t}(u)du\Delta t\right].\\
\end{array}
\end{equation}
In the above, we have implicitly applied the (reasonable) assumption that both $\pi
_{t}$ and $X_{t}^{\pi }$ are independent of the increments of the Brownian
motion sample paths, which are identically distributed over $[t,t+\Delta t]$.

Similarly, as $N\rightarrow \infty $,
\begin{equation}
\frac{1}{N}\sum_{i=1}^{N}\left( \Delta x_{t}^{i}\right) ^{2}\  \approx \  \frac{1}{N}%
\sum_{i=1}^{N}\sigma ^{2}(x_{t}^{i},u_{t}^{i}){\Delta t}\xrightarrow{\ \text{a.s.}\ }
\mathbb{E}\left[ \int_{U}\sigma ^{2}(X_{t}^{\pi },u)\pi _{t}(u)du{\Delta t}%
\right] .  \label{volatility_convergence}
\end{equation}

As we see, not only {$\Delta x_{t}^{i}$} but also {$%
(\Delta x_{t}^{i})^{2}$} are affected by  repetitive learning under the given
policy $\pi $.

Finally, as the individual state {$x_{t}^{i}$}
is an independent sample from $X_{t}^{\pi }$, we have that {$\Delta x_{t}^{i}$}
and {$(\Delta x_{t}^{i})^{2}$,} $i=1,2,\dots ,N$, are the independent samples from $%
\Delta X_{t}^{\pi }$ and $(\Delta X_{t}^{\pi })^{2}$, respectively. As a result, the law of
large numbers gives that as $N\rightarrow \infty $,
\begin{equation*}
\frac{1}{N}\sum_{i=1}^{N}\Delta x_{t}^{i}\xrightarrow{\ \text{a.s.}\ }\mathbb{E}%
\left[ \Delta X_{t}^{\pi }\right] \  \quad \  \  \text{and }\quad \frac{1}{N}%
\sum_{i=1}^{N}(\Delta x_{t}^{i})^{2}\xrightarrow{\ \text{a.s.}\ }\mathbb{E}\left[
(\Delta X_{t}^{\pi })^{2}\right] .
\end{equation*}%

This interpretation, together with (\ref{drift_convergence}) and (\ref{volatility_convergence}), {\it motivates} us to propose the {\it exploratory version} of the state dynamics, namely,
\begin{equation}\label{new_dynamics}
dX^{\pi}_t=\tilde{b}(X^{\pi}_t, \pi_t)dt+\tilde{\sigma}(X^{\pi}_t, \pi_t)dW_t,\; t>0; \quad X_{0}^{\pi}=x \in \mathbb{R},
\end{equation}
where the coefficients $\tilde{b}(\cdot,\cdot)$ and $\tilde{\sigma}(\cdot,\cdot)$ are defined as
\begin{equation}\label{drift}
\tilde{b}(y,\pi):=\int_U b\left(y,u\right)\pi(u)du,\;\;y\in\mathbb{R},\;\pi\in \mathcal{P}\left( U\right),
\end{equation}
and
\begin{equation}\label{volatility}
\tilde{\sigma}(y,\pi):=\sqrt{\int_U \sigma^2\left(y,u\right)\pi(u)du},\;\;y\in\mathbb{R},\;\pi\in \mathcal{P}\left( U\right),
\end{equation}
with $\mathcal{P}%
\left( U\right) $ being the set of density functions of probability measures on $U$ that are absolutely
continuous with respect to the Lebesgue measure.

We will call (\ref{new_dynamics}) the \textit{exploratory formulation} of the controlled state dynamics, and $\tilde{b}(\cdot,\cdot)$ and $\tilde{\sigma}(\cdot,\cdot)$ in (\ref{drift}) and (\ref{volatility}), respectively, the \textit{exploratory drift} and the \textit{exploratory volatility}.\footnote{The exploratory formulation (\ref{new_dynamics}), inspired by repetitive learning, is consistent with the notion of relaxed control in the control literature (see, for example, \cite{Nisio, Karouni, Zhou, Kurtz1, Kurtz2}). Indeed, let $f:\mathbb{R}\mapsto \mathbb{R}$ be a bounded and twice continuously differentiable function, and consider the infinitesimal generator associated to the classical controlled process (\ref{classical_state}),
\begin{equation*}
\mathbb{L}[f](x,u):=\frac{1}{2}\sigma^2(x,u)f''(x)+b(x,u)f'(x),\;\;x\in \mathbb{R},\; u\in U.
\end{equation*}
In the classical relaxed control framework, the controlled dynamics is replaced by the  six-tuple $(\Omega,\mathcal{F},\mathbb{F}=\{\mathcal{F}_t\}_{t\geq 0},\mathbb{P}, X^{\pi},\pi)$, such that $X^{\pi}_0=x$ and
\begin{equation}\label{weak_relaxed}
f(X^{\pi}_t)-f(x)-\int_0^t\int_U\mathbb{L}[f](X^{\pi}_t,u)\pi_t(u)duds,\;\;t\geq 0, \quad \text{is a}\ \mathbb{P}-\text{martingale}.
\end{equation}
It is easy to verify that our  proposed exploratory formulation (\ref{new_dynamics}) agrees with the above martingale formulation.
However, even though the mathematical formulations are equivalent, the motivations of the two are entirely different. Relaxed control was introduced to mainly deal with the existence of optimal controls, whereas the exploratory formulation here is motivated by learning and exploration in  RL. }


In a similar fashion, as $N\rightarrow \infty$,
\begin{equation}
\frac{1}{N}\sum_{i=1}^Ne^{-\rho t}r(x^i_t,u^i_t)\Delta t\xrightarrow{\ \text{a.s.}\ } \mathbb{E}\left[e^{-\rho t}\int_U r(X^{\pi}_t,u)\pi_t(u)du\Delta t\right].  \label{reward_convergence}
\end{equation}
Hence, the reward function $r$ in (\ref{classical}) needs to be modified to the {\it exploratory reward}
\begin{equation}\label{reward}
\tilde{r}\left(y,\pi\right):=\int_U r\left(y,u\right)\pi(u)du, \;\;y\in\mathbb{R},\;\pi\in \mathcal{P}\left( U\right).
\end{equation}

If, on the other hand, the model is fully known, exploration would not be needed at all and the control distributions would all degenerate to the Dirac measures,  and we would then be in the realm of the classical stochastic control. Thus, in the RL context, we need to
add a ``regularization term" to account for model uncertainty and to encourage exploration. We use Shanon's \textit{differential entropy} to measure the level
of exploration:
$$\mathcal{H}(\pi):=-\int_U\pi(u)\ln\pi(u)du,\;\;\pi\in \mathcal{P}\left( U\right).$$

We therefore introduce the following entropy-regularized relaxed stochastic control problem
{\small
\begin{equation}\label{entropy_goal}
V\left( x\right) :=\sup_{\pi \in \mathcal{A}(x)}\mathbb{E}\left[ \left.
\int_{0}^{\infty }e^{-\rho t}\left( \int_{U}r\left( X_{t}^{\pi},u\right) \pi _{t}\left( u\right) du-\lambda \int_{U}\pi _{t}(u)\ln \pi
_{t}(u)du\right) dt\right \vert X_{0}^{\pi}=x\right]
\end{equation}}%
where $\lambda>0$ is an exogenous exploration weight parameter capturing the
trade-off between exploitation (the original reward function) and
exploration (the entropy), $\mathcal{A}(x)$ is the set of the admissible control distributions (which may in general depend on $x$), and $V$ is the value function.\footnote{In the RL community, $\lambda$ is also known as the temperature parameter, which we will be using occasionally.}

The precise definition of $\mathcal{A}(x)$ depends on the specific dynamic model
under consideration and the specific problems one wants to solve, which may vary from case to case. Here, we first provide some of the ``minimal"  requirements for $\mathcal{A}(x)$.
Denote by $\mathcal{B}(U)$ the Borel algebra on $U$.
An admissible control distribution is a measure-valued (or precisely a density-function-valued) process $\pi=\{\pi
_{t},$ $t\geq 0\}$ satisfying at least the following properties:

\smallskip

(i) for each $t\geq 0$, $\pi _{t}\in \mathcal{P}(U)$ a.s.;

(ii) for each $A\in \mathcal{B}(U)$, $\{\int_A\pi _{t}(u)du,t\geq0\} $
is $\mathcal{F}_{t}$-progressively measurable;

(iii) the stochastic differential equation (SDE) (\ref{new_dynamics}) has a unique
strong solution $X^{\pi}=\{X_{t}^{\pi },t\geq 0\}$ if  $\pi $ is applied;

(iv) the expectation on the right hand side of (\ref{entropy_goal}) is finite.

\smallskip


Naturally, there could be additional requirements depending on specific problems. For the linear--quadratic control case, which will be the main focus of the paper, we define $\mathcal{A}(x)$ precisely in section 4.

\medskip

Finally, analogous to  the classical control formulation, $\mathcal{A}(x)$ contains {\it open-loop} control distributions that are measure-valued {\it stochastic processes}.
We will also consider {\it feedback} control distributions. Specifically, a {\it deterministic} mapping $\boldsymbol{\pi}(\cdot;\cdot)$ is called a feedback control (distribution) if i) $\boldsymbol{\pi}(\cdot;x)$ is a density function for each $x\in\mathbb{R}$; ii)  the following SDE (which is the system dynamics after the feedback law $\boldsymbol{\pi}(\cdot;\cdot)$ is applied)
\begin{equation}\label{new_dynamics_feedback}
dX_t=\tilde{b}(X_t, \boldsymbol{\pi}(\cdot;X_t))dt+\tilde{\sigma}(X^{\pi}_t, \ \boldsymbol{\pi}(\cdot;X_t))dW_t,\; t>0; \quad X_{0}=x \in \mathbb{R}
\end{equation}
has a unique strong solution $\{X_t;t\geq0\}$; and iii) the open-loop control
$\pi=\{\pi
_{t},$ $t\geq 0\}\in \mathcal{A}(x)$ where $\pi_{t}:=\boldsymbol{\pi}(\cdot;X_t)$. In this case, the open-loop control $\pi$ is said to be
{\it generated} from the feedback control law $\boldsymbol{\pi}(\cdot;\cdot)$ with respect to $x$.

\section{HJB Equation and Optimal Control Distributions}

We present the general procedure for solving the optimization problem (\ref{entropy_goal}). The arguments are informal and a rigorous analysis will be carried out in the next section.

To this end, applying the classical Bellman's principle of optimality, we have
\begin{equation*}
V(x)=\sup_{\pi \in \mathcal{A}(x)}\mathbb{E}\left[ \left. \int_{0}^{s}e^{-\rho
t}\left( \tilde{r}\left( X_{t}^{\pi },\pi _{t}\right) +\lambda \mathcal{H}%
\left( \pi _{t}\right) \right) dt+e^{-\rho s}V\left( X_{s}^{\pi }\right)
\right \vert X_{0}^{\pi }=x\right],\;s>0.
\end{equation*}%
Proceeding with standard  arguments, we deduce that $V$ satisfies the Hamilton-Jacobi-Bellmam (HJB) equation
\begin{equation*}
\rho v(x)=\max_{\pi \in \mathcal{P}\left( U\right) }\left( \tilde{r}(x,\pi
)-\lambda \int_{U}\pi (u)\ln \pi (u)du+\frac{1}{2}\tilde{\sigma}^{2}(x,\pi
)v^{\prime \prime }(x)\right.
\end{equation*}%
\begin{equation}
\left. +\tilde{b}(x,\pi )v^{\prime }(x)\right) ,\;\; x\in \mathbb{R}, \label{HJB}
\end{equation}%
 or
\begin{equation}\label{verif}
\rho v(x)=\max_{\pi\in \mathcal{P}\left( U\right) }\int_{U}\left( r(x,u)-\lambda \ln \pi (u)+\frac{1}{2}{%
\sigma }^{2}(x,u)v^{\prime \prime }\left( x\right) +b\left( x,u\right)
v^{\prime }(x)\right) \pi (u)du,
\end{equation}%
where $v$ denotes the generic unknown function of the equation. Recalling that $\pi\in \mathcal{P}\left( U\right)$ if and only if
\begin{equation}\label{constrained_problem}
\int_{U}\pi (u)du=1\quad \text{and}\quad \pi (u)\geq
0\  \text{a.e.}\quad  \text{on}\ U,
\end{equation}%
we can solve the (constrained) maximization problem on the right hand side of (\ref{verif}) to get a feedback control:
\begin{equation}
\boldsymbol{\pi} ^{\ast }(u;x)=\frac{\exp \left( \frac{1}{\lambda }\left( r(x,u)+\frac{1}{%
2}\sigma ^{2}\left( x,u\right) v^{\prime \prime }(x)+b\left( x,u\right)
v^{\prime }(x)\right) \right) }{\int_{U}\exp \left( \frac{1}{\lambda }\left(
r(x,u)+\frac{1}{2}\sigma ^{2}\left( x,u\right) v^{\prime \prime }(x)+b\left(
x,u\right) v^{\prime }(x)\right) \right) du}.  \label{optimizer}
\end{equation}
For each given initial state $x\in\mathbb{R}$, this feedback control in turn generates an optimal open-loop control
\begin{equation}
\pi _{t}^{\ast }:=\boldsymbol{\pi}^{\ast }(u;X_{t}^{\ast })=\frac{\exp\left( \frac{1}{\lambda }%
\left( r(X_{t}^{\ast },u)+\frac{1}{2}\sigma ^{2}(X_{t}^{\ast },u)v^{\prime
\prime }(X_{t}^{\ast })+b(X_{t}^{\ast },u)v^{\prime }(X_{t}^{\ast })\right)\right)
}{\int_{U}\exp\left( \frac{1}{\lambda }%
\left( r(X_{t}^{\ast },u)+\frac{1}{2}\sigma ^{2}(X_{t}^{\ast },u)v^{\prime
\prime }(X_{t}^{\ast })+b(X_{t}^{\ast },u)v^{\prime }(X_{t}^{\ast })\right)\right) du},  \label{optimal control}
\end{equation}%
where $\{X_{t}^{\ast }$, $t\geq 0\}$ solves (\ref{new_dynamics}) when the feedback control  law $\boldsymbol{\pi}^{\ast }(\cdot;\cdot)$ is applied and assuming that $\{\pi _{t}^{\ast },t\geq0\}\in \mathcal{A}(x).$\footnote{We stress  that the procedure described in this section, while standard, is informal. A rigorous treatment requires a precise definition of $\mathcal{A}(x)$ and a verification that indeed $\{\pi _{t}^{\ast },t\geq0\}\in \mathcal{A}(x).$
This will be carried out in the study of the linear--quadratic case in the following sections.}

Formula (\ref{optimizer}) above elicits qualitative understanding about optimal explorations. We further investigate this in the next section.

\section{The Linear--Quadratic Case}

We now focus on the family of entropy-regularized (relaxed) stochastic control  problems with
linear state dynamics and quadratic rewards, in which
\begin{equation}
b(x,u)=Ax+Bu\;\quad \text{and}\quad \;\sigma(x,u)=Cx+Du,\;\;x,u\in \mathbb{R},
\label{LQ_SDE}
\end{equation}%
where $A,B,C,D\in \mathbb{R}$, 
and
\begin{equation}
r(x,u)=-\left( \frac{M}{2}x^{2}+Rxu+\frac{N}{2}u^{2}+Px+Qu\right) ,\;\;x,u\in \mathbb{R}
\label{LQ_general_reward}
\end{equation}%
where $M\geq 0$, $N>0,$ $R,P,Q\in \mathbb{R}$.

In the classical control literature, this type of linear--quadratic (LQ) control problems is one of the most important,
not only because  it  admits elegant and simple solutions but also because more complex, nonlinear problems
can be approximated by LQ problems. As is standard with  LQ control,
we assume that the control set is unconstrained, namely, $U=\mathbb{R}$.

Fix an initial state $x\in \mathbb{R}$. For each open-loop control $\pi \in \mathcal{A}(x),$ denote its
mean and variance processes $\mu _{t},$ $\sigma _{t}^{2},t\geq 0,$ by
\begin{equation}\label{mean-variance}
\mu _{t}:=\int_{\mathbb{R}}u\pi _{t}(u)du\quad \  \  \text{and}\quad \ \ \sigma
_{t}^{2}:=\int_{\mathbb{R}}u^{2}\pi _{t}(u)du-\mu _{t}^{2}\text{ }.
\end{equation}%
Then, the state SDE (\ref{new_dynamics}) becomes
\begin{equation}
\begin{array}{rl}
dX_{t}^{\pi }=&\left( AX_{t}^{\pi }+B\mu _{t}\right) dt+\sqrt{%
C^{2}(X_{t}^{\pi })^{2}+2CDX_{t}^{\pi }\mu _{t}+D^{2}(\mu _{t}^{2}+\sigma
_{t}^{2})}\ dW_{t}\\
& \\
=&\left( AX_{t}^{\pi }+B\mu _{t}\right) dt+\sqrt{\left( CX_{t}^{\pi }+D\mu
_{t}\right) ^{2}+D^{2}\sigma _{t}^{2}}\ dW_{t},\;t>0; \quad X_{0}^{\pi } =x.
\end{array}\label{LQ_dynamics}
\end{equation}
%
Further, denote
\begin{equation*}
L(X_{t}^{\pi },\pi _{t}):=\int_{\mathbb{R}}r(X_{t}^{\pi },u)\pi
_{t}(u)du-\lambda \int_{\mathbb{R}}\pi _{t}(u)\ln \pi _{t}(u)du.
\end{equation*}%
Next, we specify the associated set of admissible controls $\mathcal{A}(x)$: $\pi\in \mathcal{A}(x)$, if

\smallskip

(i) for each $t\geq 0$, $\pi _{t}\in \mathcal{P}(\mathbb{R})$ a.s.;

(ii) for each $A\in \mathcal{B}(\mathbb{R})$, $\{\int_A\pi _{t}(u)du,t\geq0\} $
is $\mathcal{F}_{t}$-progressively measurable;

(iii) for each $t\geq 0$, $\mathbb{E}\left[
\int_{0}^{t}\left( \mu _{s}^{2}+\sigma _{s}^{2}\right) ds\right] <\infty$;

(iv) with $\{X^{\pi}_t,t\geq 0\}$ solving (\ref{LQ_dynamics}), $\liminf_{T\rightarrow \infty} e^{-\rho T}\mathbb{E}\big[\left(X^{\pi}_T\right)^2\big]=0;$

(v) with $\{X^{\pi}_t,t\geq 0\}$ solving (\ref{LQ_dynamics}), $\mathbb{E}\left[ \int_{0}^{\infty }e^{-\rho t}\left
\vert L(X_{t}^{\pi },\pi _{t})\right \vert dt\  
\right] <\infty $.

\smallskip

In the above, condition (iii) is to ensure that for any
$\pi\in \mathcal{A}(x)$, both the drift and volatility terms of (\ref%
 {LQ_dynamics}) satisfy a global Lipschitz condition and a type of linear growth condition in the state variable and, hence, the SDE (\ref%
{LQ_dynamics}) admits a unique strong solution $%
X^{\pi }$. Condition (iv) renders dynamic programming and verification technique applicable
for the model, as will be evident in the sequel. Finally, the reward is finite under condition (v).
%


We are now ready to introduce the entropy-regularized relaxed stochastic LQ problem
\begin{equation}
V(x)=\sup_{\pi \in \mathcal{A}(x)}\mathbb{E}\left[ \int_{0}^{\infty
}e^{-\rho t}\left( \int_{\mathbb{R}}r(X_{t}^{\pi },u)\pi _{t}(u)du-\lambda
\int_{\mathbb{R}}\pi _{t}(u)\ln \pi _{t}(u)du\right) dt\Big | X^{\pi}_0=x\right]  \label{LQ-V}
\end{equation}%
with $r$ as in (\ref{LQ_general_reward}) and $X^{\pi }$ as in (\ref%
{LQ_dynamics}).

In the following two subsections, we
derive explicit solutions for both cases of state-independent and state-dependent rewards.

\subsection{The case of state-independent reward }

We start with the technically less challenging case {$r(x,u)=-\left(\frac{N}{2}u^2+Qu\right)$,} namely, the reward is state (feature) independent.
In this case, the system dynamics becomes irrelevant. However, the problem is still interesting in its own right as it corresponds to the state-independent  RL problem, which is known as the continuous-armed bandit problem in the continuous time setting (\cite{con_bandit_1, con_bandit_2}).

%
Following the derivation in the previous section, the optimal feedback control in (\ref{optimizer}) reduces to
\begin{equation*}
\boldsymbol{\pi}^{\ast }\left( u;x\right) =\frac{\exp \left( \frac{1}{\lambda }\left(\left( -%
\frac{N}{2}u^{2}+Qu\right)+\frac{1}{2}(Cx+Du)^{2}v^{\prime \prime
}(x)+(Ax+Bu)v^{\prime }(x)\right) \right) }{\int_{\mathbb{R}}\exp \left(
\frac{1}{\lambda }\left(\left( -%
\frac{N}{2}u^{2}+Qu\right)+\frac{1}{2}(Cx+Du)^{2}v^{\prime
\prime }(x)+(Ax+Bu)v^{\prime }(x)\right) \right) du}
\end{equation*}%
\begin{equation}
=\frac{\exp \left( -\left( u-\frac{CDxv^{\prime \prime }(x)+Bv^{\prime }(x)-Q}{%
N-D^{2}v^{\prime \prime }(x)}\right) ^{2}/\frac{2\lambda }{N-D^{2}v^{\prime
\prime }(x)}\right) }{\int_{\mathbb{R}}\exp \left( -\left( u-\frac{%
CDxv^{\prime \prime }(x)+Bv^{\prime }(x)-Q}{N-D^{2}v^{\prime
\prime }(x)}\right) ^{2}/\frac{2\lambda }{N-D^{2}v^{\prime \prime }(x)}%
\right) \ du}.  \label{pi-feedback}
\end{equation}%

Therefore, the optimal feedback control distribution appears to be \textit{%
Gaussian}. More specifically, at any present state $x$, the agent should embark on
 exploration according to the Gaussian  distribution  with mean and variance given, respectively, by {$\frac{%
CDxv^{\prime \prime }(x)+Bv^{\prime }(x)-Q%
}{N-D^{2}v^{\prime \prime }(x)}$} and {$\frac{\lambda }{%
N-D^{2}v^{\prime \prime }(x)}$.} Note that in deriving the above, we have used that $N-D^{2}v^{\prime \prime }(x)>0$, {$x\in \mathbb{R}$,} a condition that will be justified and discussed later on.

\begin{remark}\label{Gaussian_approximate}
If we examine the derivation of  (\ref{pi-feedback}) more closely, we easily see that the
optimality of the Gaussian distribution still holds as long as the state dynamics is linear in control and the reward is quadratic in control, whereas the  dependence of both on the state can be generally {\it nonlinear}.
%
\end{remark}

Substituting (\ref{pi-feedback}) back to (\ref{HJB}), the HJB equation becomes, after straightforward calculations,
\begin{equation}\label{HJB_first_case}
\begin{array}{rl}
\rho v(x)=&\frac{\left( CDxv^{\prime \prime }(x)+Bv^{\prime }(x)-Q\right) ^{2}}{%
2\left( N-D^{2}v^{\prime \prime }(x)\right) }+\frac{\lambda }{2}\left( \ln
\left( \frac{2\pi e\lambda }{N-D^{2}v^{\prime \prime }(x)}\right) -1\right)\\
&\\
&\;\;+\frac{1}{2}C^{2}x^{2}v^{\prime \prime }(x)+Axv^{\prime }(x).
\end{array}
\end{equation}%
In general, this nonlinear equation has \textit{multiple} smooth solutions,
even among quadratic polynomials that satisfy $N-D^{2}v^{\prime \prime
}(x)>0$.
One such solution is a constant, given by
\begin{equation}
v\left( x\right)= v:=\frac{Q^2}{2\rho N}+\frac{\lambda }{2\rho }\left( \ln \frac{2\pi e\lambda }{N}%
 -1\right),  \label{constant_solution}
\end{equation}%
with the corresponding optimal feedback control distribution  (\ref{pi-feedback})\ being
\begin{equation}\label{pi_new_add}
\boldsymbol{\pi} ^{\ast }\left( u;x\right) =\frac{e^{-\left(u+\frac{Q}{N}\right)^2/ \frac{2\lambda}{N}}}{\int_{%
\mathbb{R}}e^{-\left(u+\frac{Q}{N}\right)^2/ \frac{2\lambda}{N}}du}.
\end{equation}



It turns out the right hand side of the above is independent of the current state $x$. So the optimal feedback control distribution is the same across different states.
Note that the classical LQ problem with the state-independent reward function $r(x,u)=-\left(\frac{N}{2}u^2+Qu\right)$ clearly has the optimal control $u^*=-\frac{Q}{N}$, which is also state-independent and is nothing else than the mean of the optimal Gaussian feedback control $\boldsymbol{\pi}^*$.

The following result establishes that the constant $v$ is indeed the value function $V$ and that the feedback control $\boldsymbol{\pi}^*$ defined by (\ref{pi_new_add}) is optimal. Henceforth, we denote, for notational convenience, by $\mathcal{N}(\cdot|\mu,\sigma^2 )$ the density function of a Gaussian
random variable with mean $\mu$ and variance $\sigma^2$.

\begin{theorem}\label{Prop_1} If $r(x,u)=-\left(\frac{N}{2}u^2+Qu\right)$, then
the value function in (\ref{LQ-V}) is given by
\begin{equation*}
V(x)=\frac{Q^2}{2\rho N}+\frac{\lambda }{2\rho }\left( \ln \frac{2\pi e\lambda }{N}%
 -1\right),\;\;x\in \mathbb{R},
\end{equation*}%
and the optimal feedback control distribution is Gaussian, with
$$\boldsymbol{\pi}^{\ast }(u;x)=\mathcal{N}\left(u \, \big| -\frac{Q}{N},\frac{\lambda}{N}\right).$$
Moreover, the associated optimal state process, $\{X^*_t, t\geq 0\}$, under $\boldsymbol{\pi}^{\ast}(\cdot;\cdot)$ is the unique solution of the SDE
\begin{equation}
dX_{t}^{\ast }=\left(AX_{t}^{\ast }-\frac{BQ}{N}\right)\, dt+\sqrt{\left(CX_{t}^{\ast }-\frac{DQ}{N}\right)^2+\frac{\lambda
D^{2}}{N}}\, dW_{t}, \ X_0^*=x. \label{star_SDE_1}
\end{equation}%

%
%
\end{theorem}
\begin{proof}
Let $v(x)\equiv v$ be the constant solution to the HJB equation (\ref{HJB_first_case}) defined by (\ref{constant_solution}).
Then the corresponding feedback optimizer $\boldsymbol{\pi}^{\ast }(u;x)=\mathcal{N}\left(u \, \big| -\frac{Q}{N},\frac{\lambda}{N}\right)$ follows immediately from (\ref{pi-feedback}). Let 
$\pi^{\ast}=\{\pi^{\ast}_t,t\geq0\}$  be the open-loop control generated from $\boldsymbol{\pi}^{\ast }(\cdot;\cdot)$. It is straightforward to verify that $\pi^{\ast} \in \mathcal{A}(x)$.\footnote{Since the state process is irrelevant in the current case, it is not necessary to verify the admissibility condition (iv).}

Now, for any $\pi \in \mathcal{A}(x)$
and $T\geq 0$, it follows from the HJB equation (\ref{HJB}) that
\begin{equation*}
e^{-\rho T}v=v-\int_{0}^{T}e^{-\rho t}\rho vdt
\end{equation*}%
\begin{equation*}
\leq v+\mathbb{E}\left[ \int_{0}^{T}e^{-\rho t}\left( \int_{\mathbb{R}}\left(\frac{N}{2}u^2+Qu\right)\pi _{t}(u)du+\lambda \int_{\mathbb{R}}\pi _{t}(u)\ln \pi
_{t}(u)du\right) dt\right] .
\end{equation*}%
Since $\pi \in \mathcal{A}(x)$, the dominated convergence theorem yields that, as
$T\rightarrow \infty $,
\begin{equation*}
v\geq \mathbb{E}\left[ \int_{0}^{\infty }e^{-\rho t}\left( \int_{\mathbb{R}}-%
\left(\frac{N}{2}u^2+Qu\right)\pi _{t}(u)du-\lambda \int_{\mathbb{R}}\pi _{t}(u)\ln \pi
_{t}(u)du\right) dt\right]
\end{equation*}%
and, thus, $v\geq V(x)$, for $\forall x\in \mathbb{R}$. On the other hand,
$\boldsymbol{\pi}^{\ast}$ has been derived as the maximizer for the right hand side of (\ref{HJB}); hence
\begin{equation*}
\rho v=\int_{\mathbb{R}}-\left(\frac{N}{2}u^2+Qu\right)\pi _{t}^{\ast }(u)du-\lambda \int_{%
\mathbb{R}}\pi _{t}^{\ast }(u)\ln \pi _{t}^{\ast }(u)du.
\end{equation*}%
Replacing the inequalities by equalities in the above argument and sending $%
T $ to infinity, we conclude that
\begin{equation*}
V(x)= v=\frac{Q^2}{2\rho N}+\frac{\lambda }{2\rho }\left( \ln \frac{2\pi e\lambda }{N}%
 -1\right) ,
\end{equation*}%
for $x\in \mathbb{R}$.

Finally, the exploratory dynamics equation (\ref{star_SDE_1}) follows readily from
substituting $\mu _{t}^{\ast }=-\frac{Q}{N}$ and $(\sigma _{t}^{\ast
})^{2}=\frac{\lambda}{N} $, $t\geq 0$, into (\ref{LQ_dynamics}).
\end{proof}

\medskip

It is possible to obtain {\it explicit solutions} to (\ref{star_SDE_1}) for most cases, which may be useful in designing exploration algorithms based on the theoretical results derived in this paper. We relegate this discussion about solving (\ref{star_SDE_1}) explicitly to Appendix A.

The above solution suggests that when the reward is independent of the
state, so is the optimal feedback control distribution with density $\mathcal{N}(\cdot\, |-\frac{Q}{N},\frac{\lambda}{N} )$. This is intuitive since objective (\ref%
{entropy_goal}) in this case does not explicitly distinguish between states.\footnote{Similar observation can be made for the (state-independent) pure entropy maximization
formulation, where the goal is to solve
\begin{equation}\label{only_entropy}
\sup_{\pi\in\mathcal{A}(x)}\mathbb{E}\left[-\int_0^\infty e^{-\rho t} \left(\int_U\pi_t(u)\ln \pi_t(u)du\right)dt\Big| \, X_0^{\pi}=x\right].
\end{equation}
This problem becomes relevant when $\lambda\rightarrow \infty$ in the entropy-regularized objective (\ref{LQ-V}), corresponding to the extreme case of pure exploration without considering exploitation (i.e., without maximizing any reward). To solve problem (\ref{only_entropy}), we can pointwisely maximize its integrand, leading to the state-independent optimization problem
\begin{equation}\label{Uniform}
\sup_{\pi\in \mathcal{P}(U)}\left(-\int_U\pi(u)\ln \pi(u)du\right).
\end{equation}
It is then straightforward that the optimal control distribution $\pi^*$ is, for all $t\geq 0$, the uniform distribution. This is in accordance with the traditional static setting  where uniform distribution achieves maximum entropy (\cite{Sh}).}

A remarkable feature of the derived optimal distribution $\mathcal{N}(\cdot\, |-\frac{Q}{N},\frac{\lambda}{N} )$ is that its mean coincides with the optimal control of the original, non-exploratory LQ problem, whereas the variance is determined by
the temperature parameter $\lambda$. In the context of continuous-armed bandit problem, this result stipulates that the mean is concentrated on  the current incumbent of the best arm  and the
variance is determined  by the temperature parameter.
The more weight put on the level of exploration, the more spread out the exploration becomes around the current best arm.
This type of exploration/exploitation strategies is clearly intuitive and, in turn, gives a guidance on how to actually choose the temperature parameter in practice: it is nothing else than the variance of the exploration the agent wishes to engage in
(up to a scaling factor being the quadratic coefficient of the control in the reward function).

However, we shall see in the next section that when the reward
depends on the local state, the optimal feedback control distribution genuinely 
depends on the state.

\subsection{The case of state-dependent reward}

%
%

%

We now consider the general case with the reward depending on both the control and the state, namely,
$$r(x,u)=-\left( \frac{M}{2}x^{2}+Rxu+\frac{N}{2}u^{2}+Px+Qu\right) ,\quad x,u\in \mathbb{R}.$$
%
%
%
%
We will be working with the following assumption.
\begin{assumption}\label{Assumption}
{The discount rate satisfies ${\rho} > 2A+C^{2}+\max\left(\frac{D^2R^2-2NR(B+CD)}{N},\, 0\right).$}
\end{assumption}

This assumption requires a sufficiently large discount rate, or (implicitly) a sufficiently short planning horizon. Such an assumption is standard
in infinite horizon
problems with running rewards.

Following an analogous argument as for (\ref{pi-feedback}), we deduce that a candidate optimal feedback control is given by
\begin{equation}\label{pi_new_add_2}
\boldsymbol{\pi}^{\ast }(u;x)=\mathcal{N}\left( u\, \Big | \  \
\frac{CDxv^{\prime \prime }(x)+Bv^{\prime }(x)-Rx-Q}{N-D^{2}v^{\prime \prime
}(x)}\ ,\  \frac{\lambda }{N-D^{2}v^{\prime \prime }(x)}\right).
\end{equation}%

In turn, denoting by $\mu^* (x)$ and $(\sigma^* (x))^{2}$ the mean and variance of $\boldsymbol{\pi}^*(\cdot;x)$ given above, the HJB equation (\ref{HJB}) becomes
\[
\begin{array}{rl}
\rho v(x)=&\int_{\mathbb{R}}-\left( \frac{M}{2}x^{2}+Rxu+\frac{N}{2}u^{2}+Px+Qu\right)\mathcal{N}\left( u\
\left \vert \mu^* (x),(\sigma^* (x))^{2}\right. \right) du\\
&\\
&\;+\lambda \ln \left(
\sqrt{2\pi e}{\sigma ^{\ast }}(x)\right)+v^{\prime }(x)\int_{\mathbb{R}}(Ax+Bu)\mathcal{N}\left( u\  \left \vert
\mu^* (x),(\sigma^* (x))^{2}\right. \right) du\\
&\\
&\;+\frac{1}{2}v^{\prime \prime }(x)\int_{\mathbb{R}}(Cx+Du)^{2}\mathcal{N}%
\left( u \  \left \vert \mu^* (x),(\sigma^* (x))^{2}\right. \right) du\\
=&-\frac{M}{2}x^2-\frac{N}{2}\left(\left(\frac{CDxv''(x)+Bv'(x)-Rx-Q}{N-D^2v''(x)}\right)^2+\frac{\lambda}{N-D^2v''(x)}\right)\\
&\; -(Rx+Q)\frac{CDxv''(x)+Bv'(x)-Rx-Q}{N-D^2v''(x)}-Px+\lambda\ln\sqrt{\frac{2\pi e\lambda}{N-D^2v''(x)}}\\
&\; +Axv'(x)+Bv'(x)\frac{CDxv''(x)+Bv'(x)-Rx-Q}{N-D^2v''(x)}+\frac{1}{2}C^2x^2v''(x)\\
&\; +\frac{1}{2}D^2\left(\left(\frac{CDxv''(x)+Bv'(x)-Rx-Q}{N-D^2v''(x)}\right)^2+\frac{\lambda}{N-D^2v''(x)}\right)v''(x)\\
&\; +CDxv''(x)\frac{CDxv''(x)+Bv'(x)-Rx-Q}{N-D^2v''(x)}.
\end{array}
\]%
Reorganizing, thus, the above reduces to
\begin{equation}\label{HJB_second_case}
\begin{array}{rl}
\rho v(x)=&\frac{\left(CDxv''(x)+Bv'(x)-Rx-Q\right)^2}{2(N-D^2v''(x))}+\frac{\lambda}{2}\left(\ln\left(\frac{2\pi e \lambda}{N-D^2v''(x)}\right)-1\right)\\
&\\
&\;\;+\frac{1}{2}(C^2v''(x)-M)x^2+(Av'(x)-P)x.
\end{array}
\end{equation}%
Under Assumption \ref{Assumption} and the additional condition $R^2<MN$ (which holds automatically if
$R=0$, $M>0$ and $N>0$, a standard case in the classical LQ problems), one smooth solution to the HJB equation (\ref%
{HJB_second_case}) is given by
$$v(x)=\frac{1}{2}k_{2}x^{2}+k_{1}x+k_0,$$
where\footnote{
In general, there are multiple solutions to (\ref{HJB_second_case}). Indeed, applying, for example, a generic quadratic
function ansatz $v(x)=\frac{1}{2}a_{2}x^{2}+a_{1}x+a_{0}$, $x\in \mathbb{R}$,
in (\ref{HJB_second_case}) yields the system of algebraic equations
\begin{align}
\begin{split}\label{anzats_a_2}
    \rho a_2={}& \frac{\left(a_2(B+CD)-R\right)^2}{N-a_2D^2}+a_2(2A+C^2)-M,
\end{split}\\
\begin{split}\label{ansatz_a_1}
    \rho a_1 ={}& \frac{(a_1B-Q)(a_2(B+CD)-R)}{N-a_2D^2}+a_1A-P,
\end{split}\\
    \rho a_0 ={}& \frac{(a_1B-Q)^2}{2(N-a_2D^2)}+\frac{\lambda}{2}\left(\ln\left(\frac{2\pi e \lambda}{N-a_2D^2}\right)-1\right).\label{ansatz_a_0}
\end{align}
This system has two sets of solutions (as the quadratic equation
(\ref{anzats_a_2}) has, in general, two roots), leading to two quadratic
solutions to the HJB equation (\ref{HJB_second_case}). The one given through (\ref{a_2})--(\ref{a_0}) is one of the two solutions.}
\begin{equation} \label{a_2}
\begin{array}{rl}
& \qquad \qquad \qquad k_{2}: =\frac{1}{2}\frac{(\rho-(2A+C^2)) N+2(B+CD)R-D^2M}{(B+CD)^{2} +(\rho-(2A+C^{2}))D^{2}}\\
&\\
&\mkern-18mu -\frac{1}{2}\frac{\sqrt{((\rho-(2A+C^2)) N+2(B+CD)R-D^2M)^{2}-4\left(
(B+CD)^{2}+(\rho -(2A+C^{2}))D^{2}\right)(R^2-MN) }}{(B+CD)^{2}+(\rho
-(2A+C^{2}))D^{2}},
\end{array}
\end{equation}%
\begin{equation}\label{a_1}
k_1:=\frac{P(N-k_2D^2)-QR}{k_2B(B+CD)+(A-\rho)(N-k_2D^2)-BR},
\end{equation}
and
\begin{equation}
k_{0}:=\frac{(k_1B-Q)^2}{2\rho (N-k_2D^2)}+\frac{\lambda }{2\rho }\left( \ln \left( \frac{2\pi e\lambda }{%
N-k_{2}D^{2}}\right) -1\right).  \label{a_0}
\end{equation}%

\bigskip

For this particular solution,  given by $v(x)$ above, we
can verify that  $k_{2}<0$, due to Assumption \ref{Assumption} and $R^2<MN$. Hence, $v$ is concave, a property that is essential
in proving that it is actually the value function.\footnote{Under Assumption \ref{Assumption} and $R^2<MN$, the HJB equation has an additional quadratic solution, which however is {\it convex}.}  On the other hand, $N-D^{2}v^{\prime \prime }(x)=N-k_2D^2>0$, ensuring that $k_0$ is well defined.

\smallskip

Next, we state one of the main results of this paper.

\begin{theorem}\label{Theorem_general}
Suppose 
the reward function is given by
$$r(x,u)=-\left( \frac{M}{2}x^{2}+Rxu+\frac{N}{2}u^{2}+Px+Qu\right),$$
with $M\geq 0$, $N>0$, $R,Q,P\in \mathbb{R}$ and $R^2<MN$. Furthermore, suppose that Assumption \ref{Assumption} holds.
Then, the value function in (\ref{LQ-V}) is given by
\begin{equation}
V(x)=\frac{1}{2}k_{2}x^{2}+k_{1}x+k_0,\;\;x\in \mathbb{R},  \label{V_2nd_case}
\end{equation}%
where $k_{2}$, $k_{1}$ and $k_0$ are as in (\ref{a_2}), (\ref{a_1}) and (\ref%
{a_0}), respectively.
Moreover,
the optimal feedback control is Gaussian, with its density function given by
\begin{equation}
\boldsymbol{\pi}^{\ast }(u;x)=\mathcal{N}\left( u\  \left \vert \frac{(k_{2}(B+CD)-R)x+k_1B-Q}{%
N-k_{2}D^{2}}\ ,\  \frac{\lambda }{N-k_{2}D^{2}}\right. \right).
\label{Gaussian_verified}
\end{equation}%
Finally, the associated optimal state process $\{X^*_t, t\geq 0\}$ under $\boldsymbol{\pi}^{\ast }(\cdot;\cdot)$
is the unique solution of the SDE
{\small $$dX_{t}^{\ast }=\left(\left( A+\frac{B(k_{2}(B+CD)-R)}{N-k_{2}D^{2}}\right) X_{t}^{\ast}+\frac{B(k_1B-Q)}{N-k_2D^2}\right)\, dt$$
\begin{equation}
+\sqrt{\left(\left( C+\frac{D(k_{2}(B+CD)-R)}{N-k_{2}D^{2}}\right) X_{t}^{\ast
}+\frac{D(k_1B-Q)}{N-k_2D^2}\right)^2+\frac{\lambda D^{2}}{N-k_{2}D^{2}}}\, dW_{t},\;X_{0}^{\ast }=x.
\label{second_state_dynamics}
\end{equation}}%

%
%
%
\end{theorem}

A proof of this theorem follows essentially the same idea as that of Theorem \ref{Prop_1}, but it is  more technically involved, mainly for
verifying the admissibility of the candidate optimal control. To ease the presentation, we defer it
to Appendix B.

\begin{remark}
As in the state-independent case (see Appendix A), the solution to the SDE (\ref{second_state_dynamics}) can be expressed through the Doss-Saussman transformation if $D\neq 0$.

Specifically,
if $C+\frac{D(k_{2}(B+CD)-R)}{N-k_{2}D^{2}}\neq 0$, then
\[
X_{t}^{\ast }=F(W_t,Y_t) ,\quad t\geq 0,
\]
where the function $F$ is given by
$$F(z,y)=\frac{\sqrt{\tilde{D}}}{|\tilde{C_1}|}\sinh\left(|\tilde{C_1}|z+\sinh^{(-1)}
\left(\frac{|\tilde{C_1}|}{\sqrt{\tilde{D}}}\left(y+\frac{\tilde{C_2}}{\tilde{C_1}}\right)\right)\right)
-\frac{\tilde{C_2}}{\tilde{C_1}},$$
and the process $Y_{t}$, $t\geq 0$, is the unique pathwise solution to the random
ODE
\begin{equation*}
\frac{dY_{t}}{dt}=\frac{\tilde{A}F(W_t,Y_t)+\tilde{B}-\frac{\tilde{C_1}}{2}\left(\tilde{C_1} F(W_t,Y_t)+\tilde{C_2}\right)}{\frac{%
\partial }{\partial y}F(z,y)\big|_{z=W_t, y=Y_t}},\;Y_{0}=x,
\end{equation*}%
with $\tilde{A}:=A+\frac{B(k_{2}(B+CD)-R)}{N-k_{2}D^{2}}$, $\tilde{%
B}:=\frac{B(k_1B-Q)}{N-k_2D^2}$, $\tilde{C_1}:=C+\frac{D(k_{2}(B+CD)-R)}{N-k_{2}D^{2}}$, \\
\vspace{-5pt}
\begin{flushleft}
$\tilde{C_2}:=\frac{D(k_1B-Q)}{N-k_2D^2}$ and $\tilde{D}:=\frac{\lambda D^{2}}{N-k_{2}D^{2}}$.
\end{flushleft}

If $C+\frac{D(k_{2}(B+CD)-R)}{N-k_{2}D^{2}}= 0$ and $\tilde{A}\neq0$, then it follows from direct computation that
\[
X_{t}^{\ast }=xe^{\tilde{A}t}-\frac{\tilde{B}}{\tilde{A}}(1-e^{\tilde{A}t})+\sqrt{\tilde{C_1}^2+\tilde{D}}\int_{0}^{t}e^{\tilde{A}%
(t-s)}dW_{s},\quad t\geq 0.
\]
We leave the detailed derivations to the interested readers.
\end{remark}

The above results demonstrate that, for the general state and control dependent reward case, the optimal actions over $\mathbb{R}$ also depend on the current state $x$,
which are selected according to a state-dependent Gaussian distribution (\ref%
{Gaussian_verified}) with a state-independent variance $\frac{\lambda }{N-k_{2}D^{2}}$. Note that if $D\neq0$, then $\frac{\lambda }{N-k_{2}D^{2}}<\frac{\lambda }{N}$ (since $k_2<0$). Therefore, the exploration variance in the general state-dependent case is {\it strictly smaller} than $\frac{\lambda}{N}$, the one in the state-independent case. Recall that
$D$ is the coefficient of the control in the diffusion term of the state dynamics,
generally representing the level of randomness of the environment.\footnote{For example, in the Black--Scholes market, $D$ is the volatility parameter of the underlying stock.} Therefore, volatility impacting actions
reduce the need for exploration.
Moreover,  the greater $D$ is, the smaller the exploration variance becomes, indicating that even less exploration is required.
As a result, the need for exploration is further reduced if an action has a greater impact on the volatility of
the system dynamics. This hints that a more volatile environment renders more
learning opportunities.

On the other hand, the mean of the Gaussian distribution does not explicitly
depend on $\lambda $. The implication is that the agent should concentrate on
the most promising region in the action space while randomly selecting actions
to interact with the unknown environment. It is intriguing that the entropy-regularized
RL formulation separates the exploitation from exploration, respectively through the
mean and variance of the resulting optimal Gaussian distribution.

\begin{remark}
It should be noted that it is the optimal {\bf feedback} control distribution, {\it not} the open-loop control generated from the feedback control, that has
the Gaussian distribution. More precisely, $\boldsymbol{\pi}^*(\cdot;x)$ defined by (\ref{Gaussian_verified}) is Gaussian for each and every $x$, but the measure-valued process with the density function
\begin{equation}
\pi^{\ast }_t(u):=\mathcal{N}\left( u\  \left \vert \frac{(k_{2}(B+CD)-R)X^{\ast }_t+k_1B-Q}{%
N-k_{2}D^{2}}\ ,\  \frac{\lambda }{N-k_{2}D^{2}}\right. \right) ,\;\;t\geq 0,
\label{Gaussian_verified-ol}
\end{equation}
where $\{X^{\ast }_t,t\geq0\}$ is the solution of the exploratory dynamics under the
feedback control $\boldsymbol{\pi}^*(\cdot;\cdot)$ with any fixed initial state, say, $X^{\ast }_0=x_0$, is
in general {\it not} Gaussian for any $t>0$. The reason is that for each $t>0$, the right hand side
of (\ref{Gaussian_verified-ol}) is a composition of the Gaussian density function and
a random variable $X^{\ast }_t$ whose distribution is unknown. We stress that the Gaussian property for feedback control is more important and relevant in the RL context, as it stipulates that at any given state, if one undertakes  exploration then she should follow Gaussian. The open-loop control $\{\pi^{\ast }_t,t\geq0\}$, generated from the Gaussian feedback control, 
is just what the agent would end up if she follows Gaussian exploration at every state.
\end{remark}

Finally, as noted earlier (see Remark \ref{Gaussian_approximate}), the optimality of the Gaussian distribution is
 still valid for problems with dynamics
$$dx_t=\left(A(x_t)+B(x_t)u_t\right)\, dt+\left(C(x_t)+D(x_t)u_t\right)\, dW_t,$$
and reward function in the form $r(x,u)=r_2(x)u^2+r_1(x)u+r_0(x)$, where the functions $A,B,C,D,r_2, r_1$ and $r_0$ are possibly nonlinear (pending some additional assumptions for the verification arguments to hold).

\section{The Cost and Effect of Exploration}

Motivated by the necessity of
exploration facing the typically unknown environment in an RL setting, we have formulated and analyzed a new class of stochastic control
problems that combine entropy-regularized criteria and relaxed controls. We
have also derived closed-form solutions and presented verification results for
the important class of LQ problems.
A natural question arises, namely, how to quantify the cost and effect of the exploration. This can be done by comparing our results to the ones for the
classical stochastic LQ problems, which have neither entropy regularization
nor control relaxation.

We carry out this comparison  analysis next.

\subsection{The classical LQ problem}

We first briefly recall the classical stochastic LQ control problem in an infinite horizon with
discounted reward. Let $\{{W}_{t}$, $t\geq 0\}$ be a standard Brownian
motion defined on the filtered probability space $({\Omega},{%
\mathcal{F}},\{ {\mathcal{F}}_{t}\}_{t\geq 0},{%
\mathbb{P}})$ that satisfies the usual conditions. The controlled state process $%
\{x^u_{t},t\geq 0\}$ solves
\begin{equation}\label{classical_LQ_dynamics}
dx^u_{t}=\left( Ax^u_{t}+Bu_{t}\right) \,dt+\left( Cx^u_{t}+Du_{t}\right) d{W}%
_{t}\,,\quad t\geq 0,\;\;x^u_{0}=x,
\end{equation}%
with given constants $A,B,C$ and $D,$ and the process $\{u_{t}$, $%
t\geq 0\}$ being a (classical, non-relaxed) control.

The value function is defined as in (\ref{classical}),
\begin{equation}\label{LQ_classical_V}
V^{\text{cl}}(x):=\sup_{u\in \mathcal{A}^{\text{cl}}(x)}\mathbb{E}\left[ \left.
\int_{0}^{\infty }e^{-\rho t}r(x^u_{t},u_{t})\ dt\right \vert \ x^u_{0}=x\right],
\end{equation}%
for $x\in \mathbb{R}$, where the reward function $r(\cdot ,\cdot )$ is given
by (\ref{LQ_general_reward}). Here,
the admissible set $\mathcal{A}^{\text{cl}}(x)$ is defined as follows: $u\in \mathcal{A}^{\text{cl}}(x)$ if

\medskip

(i) $\{u_t, t\geq 0\}$ is $\mathcal{F}_t$-progressively measurable;

(ii) for each $t\geq 0$, $\mathbb{E}\left[\int_0^t(u_s)^2\, ds\right]<\infty$;

(iii) with $\{x^u_t, t\geq 0\}$ solving (\ref{classical_LQ_dynamics}), $\liminf_{T\rightarrow \infty} e^{-\rho T}\mathbb{E}\big[(x^u_T)^2\big]=0$;

(iv) with $\{x^u_t, t\geq 0\}$ solving (\ref{classical_LQ_dynamics}), $\mathbb{E}\big[\int_0^\infty e^{-\rho t}|r(x^u_t,u_t)|\, dt\big]<\infty.$

\medskip

The associated HJB equation is
\begin{equation*}
\rho {w}(x)=\max_{u\in \mathbb{R}}\left( r(x,u)+\frac{1}{2}(Cx+Du)^{2}{w}%
^{\prime \prime }(x)+{(Ax+Bu)w}^{\prime }(x)\right)
\end{equation*}%
\begin{equation*}
=\max_{u\in \mathbb{R}}\left( -\frac{1}{2}\left( N-D^{2}{w}^{\prime \prime
}(x)\right) u^{2}+\left( {CDxw}^{\prime \prime }(x)+B{w}^{\prime
}(x)-Rx-Q\right) u\right)
\end{equation*}%
\begin{equation*}
+\frac{1}{2}(C^{2}{w}^{\prime \prime}(x)-M)x^2 +(A{w}^{\prime }(x)-P)x
\end{equation*}%
\begin{equation}
=\frac{\left( CDx{w}^{\prime \prime }(x)+B{w}^{\prime }(x)-Rx-Q\right) ^{2}}{%
2(N-D^{2}{w}^{\prime \prime }(x))}+\frac{1}{2}({C^{2}w}^{\prime \prime}(x)-M)x^2+(A{%
w}^{\prime }(x)-P)x,  \label{classical_HJB}
\end{equation}%
with the maximizer being, provided that $%
N-D^{2}{w}^{\prime \prime }(x)>0$,
\begin{equation}\label{classical_optimal_feedback}
\boldsymbol{u}^{\ast }(x)=\frac{CDxw^{\prime \prime }(x)+Bw^{\prime }(x)-Rx-Q}{%
N-D^{2}w^{\prime \prime }(x)},\;\;x\in \mathbb{R}.
\end{equation}%
The standard verification argument then deduces that $\boldsymbol{u}$ is the optimal feedback control. 

In the next section, we will establish a solvability equivalence between the entropy-regularized relaxed LQ problem and the classical one.

\subsection{Solvability equivalence of classical and exploratory \\ problems}


Given a reward function $r(\cdot ,\cdot )$ and a classical controlled
process (\ref{classical_state}), the relaxed formulation (\ref{new_dynamics}%
) under the entropy-regularized  objective is, naturally, a technically more challenging problem, compared to its classical counterpart.

In this section, we show that there is actually a solvability equivalence between the
exploratory and the classical  stochastic LQ problems, in the sense that the value function and optimal control of one problem lead
directly to those of the other.
Such equivalence enables us to readily establish the convergence result as
the exploration weight $\lambda $ decays to zero. Furthermore, it makes it possible to
quantify the exploration cost, which we introduce in the sequel.

\begin{theorem}\label{Theorem_equivalence}
The following two statements (a) and (b) are
equivalent.
\begin{description}
\item[(a)] \ The function $v(x)=\frac{1}{2}\alpha _{2}x^{2}+\alpha
_{1}x+\alpha _{0}+\frac{\lambda }{2\rho }\left( \ln \left( \frac{2\pi
e\lambda }{N-\alpha _{2}D^{2}}\right) -1\right) $, $x\in \mathbb{R}$, with $\alpha_0,\alpha_1\in \mathbb{R}$ and
$\alpha _{2}<0$, is the value function of the exploratory problem (\ref{LQ-V}) and the
corresponding optimal feedback control is
\begin{equation*}
\boldsymbol{\pi}^{\ast }(u;x)=\mathcal{N}\left( u\  \left \vert \frac{(\alpha_{2}(B+CD)-R)x+\alpha_1B-Q}{%
N-\alpha_{2}D^{2}}\ ,\  \frac{\lambda }{N-\alpha_{2}D^{2}}\right. \right).
\end{equation*}%

\item[(b)] \ The function ${w}(x)=\frac{1}{2}\alpha _{2}x^{2}+\alpha
_{1}x+\alpha _{0}$, $x\in \mathbb{R}$, with $\alpha_0,\alpha_1\in \mathbb{R}$ and $\alpha _{2}<0$, is the value function of the classical problem (\ref{LQ_classical_V})
and the corresponding optimal feedback control is
\begin{equation*}
\boldsymbol{u}^{\ast }(x)=\frac{(\alpha _{2}(B+CD)-R)x+\alpha _{1}B-Q}{%
N-\alpha _{2}D^{2}}.
\end{equation*}%
\end{description}
\end{theorem}
\begin{proof}
See Appendix C.
\end{proof}

\bigskip

The above equivalence between statements (a) and (b) yields that if one problem is solvable, so is the other; and
conversely, if one is not solvable, neither is the other.

\subsection{Cost of exploration}

We define the exploration
cost for a general RL problem to be the difference
between the discounted accumulated rewards following the corresponding optimal  {\it open-loop} 
controls under the classical objective (\ref{classical}) and the exploratory
objective (\ref{entropy_goal}), net of the value of the entropy.  Note that the
solvability equivalence established in the previous subsection  is important for this definition, not least because the cost is well defined only if both the classical and
the exploratory problems are solvable.

Specifically, let the classical maximization problem (\ref{classical})
with the state dynamics (\ref{classical_state}) have the value function $V^{\text{cl}}(\cdot)$ and optimal strategy $\{u_{t}^{\ast },t\geq 0\}$,
and  the corresponding exploratory  problem have the value function $V(\cdot
)$ and optimal control distribution $\{\pi _{t}^{\ast },%
t\geq 0\}$.
Then, we define the \textit{exploration cost} as
\begin{equation}
\mathcal{C}^{u^{\ast },\pi ^{\ast }}(x):=V^{\text{cl}}(x)-\left( V(x)+{\lambda}\mathbb{E}\left[
\left. \int_{0}^{\infty }e^{-\rho t}\left( \int_{{U}}\pi _{t}^{\ast }(u)\ln
\pi _{t}^{\ast }(u)du\right) \,dt\right \vert X_{0}^{\pi^*}=x\right] \right) ,
\label{exploration_cost}
\end{equation}%
for $x\in \mathbb{R}$.

The term in the parenthesis represents the total discounted rewards incurred
by $\pi ^{\ast }$ after taking out the contribution of the entropy term to the value
function $V(\cdot )$ of the exploratory problem. The exploration cost hence
measures the best outcome due to the explicit inclusion
of exploratory strategies in the entropy-regularized  objective,
relative to the benchmark $V^{\text{cl}}(\cdot)$ which is the best possible objective
value should the  model
be {\it a priori} fully known.

We next compute the exploration cost for the LQ case.
As we show, this cost is surprisingly simple: it  depends only on two ``agent-specific" parameters: the temperature
parameter $\lambda $ and the discounting parameter $\rho $.

\begin{theorem}\label{exploration_cost_theorem}
Assume that statement (a) (or equivalently, (b)) of Theorem \ref{Theorem_equivalence} holds. Then, the exploration cost for the stochastic LQ problem is
\begin{equation}\label{LQ_cost_theorem}
\mathcal{C}^{u^{\ast },\pi ^{\ast }}(x)
=\frac{%
\lambda }{2\rho },\;\;\mbox{for $x\in \mathbb{R}$}.
\end{equation}%
\end{theorem}
\begin{proof}
Let $\{\pi^{\ast }_t,t\geq0\}$ be the open-loop control generated by the feedback control $\boldsymbol{\pi}^{\ast }$ given in statement (a) with respect to the initial state $x$, namely,
\[
\pi^{\ast }_t(u)=\mathcal{N}\left( \ u\ \Big|  \frac{(\alpha_{2}(B+CD)-R)X_{t}^{\ast }+\alpha _{1}B-Q}{N-\alpha _{2}D^{2}}\ ,\  \frac{%
\lambda }{N-\alpha _{2}D^{2}}\right)
\]
where $\{X_{t}^{\ast }, t\geq 0\}$ is the
associated state process of the exploratory problem, starting from the state $x$, when $\boldsymbol{\pi}^{\ast }$ is applied. 
Then, it is straightforward to calculate 
\[ \int_{\mathbb{R}}\pi^{\ast }_t(u)\ln \pi^{\ast }_t(u)du=-\frac{1}{2}\ln \left( \frac{2\pi e\lambda }{N-\alpha _{2}D^{2}}\right).\]
The desired result now follows immediately from the general definition in (\ref{exploration_cost}) and the expressions of $V(\cdot )$ in (a) and $V^{\text{cl}}(\cdot)$ in (b).
\end{proof}

 \medskip

In other words,  the exploration cost for
stochastic LQ problems can be completely pre-determined by the learning
agent through choosing her individual parameters $\lambda $ and $\rho $, since the cost relies neither on the specific (unknown) linear state dynamics, nor on the quadratic reward structure.

Moreover, the exploration cost (\ref{LQ_cost_theorem}) depends on $\lambda $ and $\rho $ in a rather
intuitive way:  it increases as $\lambda $ increases, due to more
emphasis placed on exploration, or as $\rho $ decreases, indicating an
effectively longer horizon for exploration.\footnote{The connection between
 a discounting parameter and an effective length of time horizon is  well known in
the discrete time discounted reward formulation $\mathbb{E}[\sum_{t\geq
0}\gamma ^{t}R_{t}]$ for classical Markov Decision Processes (MDP) (see, among others,
\citet{Derman}). This infinite horizon discounted problem can be viewed as an
undiscounted, finite horizon  problem with a random termination time $T$ that
is geometrically distributed with parameter $1-\gamma $.
Hence, an effectively longer horizon with mean $\frac{1}{1-\gamma }$
is applied to the optimization problem as $\gamma $ increases. Since a
smaller $\rho $ in the continuous time objective (\ref{classical}) or (\ref%
{entropy_goal}) corresponds to a larger $\gamma $ in the discrete time
objective, we can see the similar effect of a decreasing $\rho $ on the
effective horizon of continuous time problems.}

\subsection{Vanishing exploration}


Herein, the exploration weight $\lambda$ has been taken as an exogenous parameter reflecting the level of exploration desired by the learning agent. The smaller this parameter is, the more emphasis is placed on exploitation. When this parameter is sufficiently close to zero, the exploratory formulation is sufficiently close to the  problem without exploration. Naturally,
a desirable result is that if the exploration weight $\lambda$ goes to zero, then the entropy-regularized LQ
problem would converge to its classical counterpart. The following result makes this precise.

\begin{theorem}\label{convergence_to_Dirac}
Assume that statement (a) (or equivalently, (b)) of Theorem \ref{Theorem_equivalence} holds. Then, for each $x\in \mathbb{R}$, 
$$\lim_{\lambda \rightarrow 0}\boldsymbol{\pi}^{\ast}(\cdot;x)=\delta_{\boldsymbol{u}^{\ast}(x)}(\cdot) \;\;\mbox{ weakly.}$$
Moreover, 
for each $%
x\in \mathbb{R}$,
$$\lim_{\lambda \rightarrow 0}|V(x)-V^{\text{cl}}(x)|=0.$$
\end{theorem}

\begin{proof}
The weak convergence of the feedback controls is due to the explicit forms of $\boldsymbol{\pi}^{\ast}$  and $\boldsymbol{u}^{\ast}$ in statements (a) and (b), and the fact that $\alpha_1$, $\alpha_2$ are independent of $\lambda$. The pointwise convergence of the value functions follows easily from the forms of $V(\cdot )$ and $%
V^{\text{cl}}(\cdot )$, together with the fact that
\[ \lim_{\lambda \rightarrow 0}\frac{\lambda }{2\rho }\left( \ln \left( \frac{2\pi
e\lambda }{N-\alpha _{2}D^{2}}\right) -1\right) =0.
\]
\end{proof}

\section{Conclusions}

This paper approaches RL from a stochastic control perspective. Indeed,
control and RL both deal with the problem of managing dynamic and stochastic systems
by making the best use of  available information. However, as a recent survey paper
\cite{Recht} points out, ``...{\it That the RL and
control communities remain practically disjoint has led to the co-development of vastly different
approaches to the same problems}...." It is our view that communication and exchange of ideas between the two fields are of paramount
importance to the progress of both fields, for an old idea from one field may well be
a fresh one to the other. The continuous-time relaxed stochastic control formulation employed in this paper exemplifies such a vision.

The main contributions of this paper are {\it conceptual} rather than {\it algorithmic}: casting the RL problem in a continuous-time setting and with the aid of stochastic control and stochastic calculus, we interpret and explain why the Gaussian distribution is best for exploration
in RL. This finding is independent of the specific parameters of the underlying dynamics and reward function structure, as long as the dependence on actions
is linear in the former and quadratic in the latter. The same can be said about other main results of the paper, such as the separation between exploration and exploitation
in the mean and variance of the resulting Gaussian distribution, and the cost of exploration. The explicit forms of the derived optimal Gaussian distributions do indeed depend on the
model specifications which are unknown in the RL context. With regards to implementing RL algorithms based on our results for LQ problems, we can either do it in continuous time and space directly following, for example, \cite{Doya}, or modify the problem into an MDP one by discretizing the time, and then learn the parameters of the optimal Gaussian distribution following standard RL procedures (e.g. the so-called $Q$-learning).
For that, our results may again be useful: they suggest that we only need to learn among the class of
simpler Gaussian policies, i.e., $\pi=\mathcal{N}(\cdot\, | \theta_1x+\theta_2, \phi)$ (cf. (\ref{Gaussian_verified})), rather than generic (nonlinear) parametrized Gaussian policy $\pi_{\theta,\phi}=\mathcal{N}(\cdot\, | \theta(x), \phi(x))$. We expect that this  simpler functional  form can considerably increase the learning speed.

\begin{appendix}
\section*{Appendix A: Explicit Solutions to (\ref{star_SDE_1})}
For a range of parameters, we derive explicit solutions to SDE (\ref{star_SDE_1}) satisfied by the optimal state process $\{X^*_t,t\geq 0\}$.

If $D=0$, the SDE (\ref{star_SDE_1}) reduces to
$$dX^*_t=\left(AX^*_t-\frac{BQ}{N}\right)\,dt+|C|\; |X^*_t|\, dW_t,\quad X^*_0=x.$$
If $x\geq 0$ and $BQ\leq 0$, the above equation has a nonnegative solution given by
$$X^*_t=xe^{\left(A-\frac{C^2}{2}\right)t+|C|W_t}-\frac{BQ}{N}\int_0^te^{\left(A-\frac{C^2}{2}\right)(t-s)+|C|(W_t-W_s)}ds.$$
If $x\leq 0$ and $BQ\geq 0$, it has a nonpositive solution
$$X^*_t=xe^{\left(A-\frac{C^2}{2}\right)t-|C|W_t}
-\frac{BQ}{N}\int_0^te^{\left(A-\frac{C^2}{2}\right)(t-s)-|C|(W_t-W_s)}ds.$$
These two cases cover the special case when $Q=0$ which is standard in the LQ control formulation. We are unsure if there is an explicit solution when neither of these assumptions is satisfied (e.g. when $x\geq 0$ and $BQ> 0$).

If $C=0$, the SDE (\ref{star_SDE_1}) becomes
$$dX_{t}^{\ast }=\left(AX_{t}^{\ast }-\frac{BQ}{N}\right)\, dt+%
\frac{|D|}{N}\sqrt{Q^2+\lambda N }\, dW_{t},$$ and its unique
solution is given by
\[
X_{t}^{\ast }=xe^{At}-\frac{BQ}{AN}(1-e^{At})+\frac{|D|}{N}\sqrt{Q^2+\lambda N}\int_{0}^{t}e^{A(t-s)}dW_{s},\quad t\geq 0,
\]
if $A\neq0$, and
\[ X_{t}^{\ast }=x-\frac{BQ}{N}t+\frac{|D|}{N}\sqrt{Q^2+\lambda N }W_t, \quad t\geq 0,
\]
if $A=0$.

If $C\neq 0$ and $D\neq0$, then the diffusion coefficient of SDE (\ref{star_SDE_1}) is $C^2$ in the unknown, with the first and second order derivatives being bounded. Hence, (\ref{star_SDE_1}) can be solved explicitly using the
Doss-Saussman transformation (see, for example, \cite{KS}, pp 295-297).
This transformation uses the ansatz
\begin{equation}\label{DS_appendix}
X_{t}^{\ast }(\omega )=F(W_{t}(\omega ),Y_{t}(\omega )),\quad t\geq 0,\;\omega \in \Omega
\end{equation}%
for some deterministic function $F$ and an adapted process $Y_{t}$, $t\geq 0$, solving a random ODE.
Applying It\^o's formula to (\ref{DS_appendix}) and using the dynamics in (\ref{star_SDE_1}), we deduce that $F$ solves, for each fixed $y$, the ODE
\begin{equation}
{\frac{\partial F}{\partial z}=\sqrt{\left(CF(z,y)-\frac{DQ}{N}\right)^2+\frac{\lambda D^{2}}{N}},\quad F(0,y)=y.}
\label{DS_u}
\end{equation}%
Moreover, $Y_{t}$, $t\geq 0$, is the unique pathwise solution to the random
ODE
\begin{equation}
\frac{d}{dt}Y_{t}(\omega )=G(W_{t}(\omega ),Y_{t}(\omega )),\quad
Y_{0}(\omega )=x,  \label{DS_Y}
\end{equation}%
where
$$G(z,y)=\frac{AF(z,y)-\frac{BQ}{N}-\frac{C}{2}\left(CF(z,y)-\frac{DQ}{N}\right)}{\frac{%
\partial }{\partial y}F(z,y)}.$$
It is easy to verify that both equations (\ref{DS_u}) and (\ref{DS_Y}) have a unique solution. Solving (\ref{DS_u}), we obtain
$$F(z,y)=\sqrt{\frac{\lambda}{N} }\left \vert
\frac{D}{C}\right \vert \sinh \left( |C|z+\sinh ^{(-1)}\left( {\sqrt{%
\frac{N}{\lambda} }}\left \vert \frac{C}{D}\right \vert \left(y-\frac{DQ}{CN}\right) \right)\right)+\frac{DQ}{CN}.$$
This, in
turn, leads to the explicit expression of the function $G(z,y)$.

\section*{Appendix B: Proof of Theorem 4}

Recall that the function $v$, where $v(x)=\frac{1}{2}k_{2}x^{2}+k_{1}x+k_0,\;\;x\in \mathbb{R}$,
where $k_{2}$, $k_{1}$ and $k_0$ are defined by (\ref{a_2}), (\ref{a_1}) and (\ref%
{a_0}), respectively, satisfies
the HJB equation (\ref{HJB}).

Throughout this proof we fix the initial state $x\in\mathbb{R}$. 
Let $\pi\in \mathcal{A}(x)$ and $X^{\pi}$ be the associated state process solving (\ref{LQ_dynamics}) with $\pi$ being used. Let $T>0$ be arbitrary.
Define the stopping times $\tau _{n}^{\pi }:=\{t\geq 0:\int_{0}^{t}\left(
e^{-\rho t}v^{\prime }(X_{t}^{\pi })\tilde{\sigma}(X_{t}^{\pi },\pi
_{t})\right) ^{2}dt\geq n\}$, for $n\geq 1$. Then, It\^{o}'s
formula yields
\begin{equation*}
e^{-\rho (T\wedge \tau _{n}^{\pi })}v(X_{T\wedge \tau _{n}^{\pi }}^{\pi
})=v(x)+\int_{0}^{T\wedge \tau _{n}^{\pi }}e^{-\rho t}\Big( -\rho
v(X_{t}^{\pi })+\frac{1}{2}v^{\prime \prime }(X_{t}^{\pi })\tilde{\sigma}%
^{2}(X_{t}^{\pi },\pi _{t})
\end{equation*}%
\begin{equation*}
+v^{\prime }(X_{t}^{\pi })\tilde{b}(X_{t}^{\pi },\pi _{t})\Big) \ dt+\int_{0}^{T\wedge \tau _{n}^{\pi }}e^{-\rho t}v^{\prime
}(X_{t}^{\pi })\tilde{\sigma}(X_{t}^{\pi },\pi _{t})\ dW_{t}.
\end{equation*}%
Taking expectations, using that $v$ solves the HJB equation (\ref{HJB}) and that $\pi$ is in general suboptimal yield
\begin{equation*}
\mathbb{E}\left[ e^{-\rho (T\wedge \tau _{n}^{\pi })}v(X_{T\wedge \tau
_{n}^{\pi }}^{\pi })\right]
\end{equation*}%
\begin{equation*}
=v(x)+\mathbb{E}\left[ \int_{0}^{T\wedge \tau _{n}^{\pi }}e^{-\rho t}\left(
-\rho v(X_{t}^{\pi })+\frac{1}{2}v^{\prime \prime }(X_{t}^{\pi })\tilde{%
\sigma}^{2}(X_{t}^{\pi },\pi _{t})+v^{\prime }(X_{t}^{\pi })\tilde{b}%
(X_{t}^{\pi },\pi _{t})\right)\, dt \right]
\end{equation*}%
\begin{equation*}
\leq v(x)-\mathbb{E}\left[ \int_{0}^{T\wedge \tau _{n}^{\pi }}e^{-\rho
t}\left( \tilde{r}(X_{t}^{\pi },\pi _{t})-\lambda \int_{\mathbb{R}}\pi
_{t}(u)\ln \pi _{t}(u)du\right) dt\right].
\end{equation*}%

Classical results yield $\mathbb{E}\left[ \sup_{0\leq t\leq T}|X_{t}^{\pi }|^{2}\right] \leq
K(1+x^{2})e^{KT}$, for some constant $K>0$ independent of $n$ (but dependent on $T$ and the model coefficients). Sending $n\rightarrow \infty $, we deduce that
\begin{equation*}
\mathbb{E}\left[ e^{-\rho T}v(X_{T}^{\pi })\right]  \leq v(x)-\mathbb{E}\left[
\int_{0}^{T}e^{-\rho t}\left( \tilde{r}(X_{t}^{\pi },\pi _{t})-\lambda \int_{%
\mathbb{R}}\pi _{t}(u)\ln \pi _{t}(u)du\right) dt\right] ,
\end{equation*}%
where we have used the dominated convergence theorem and that  $\pi \in
\mathcal{A}(x)$.

Next, we recall the admissible condition $\liminf_{T\rightarrow \infty }e^{-\rho T}\mathbb{E}\left[
(X_{T}^{\pi })^{2}\right] =0.$ This, together with the fact that $k_{2}<0$, lead to $\limsup_{T\rightarrow \infty }\mathbb{E}\left[ e^{-\rho T}v(X_{T}^{\pi })\right] = 0$.
Applying the dominated convergence theorem once more yields
\begin{equation*}
v(x)\geq \mathbb{E}\left[ \int_{0}^{\infty }e^{-\rho t}\left( \tilde{r}%
(X_{t}^{\pi },\pi _{t})-\lambda \int_{\mathbb{R}}\pi _{t}(u)\ln \pi
_{t}(u)du\right) dt \right] ,
\end{equation*}%
for each $x\in \mathbb{R}$ and $\pi \in \mathcal{A}(x)$. Hence, $%
v(x)\geq V(x)$, for all $x\in \mathbb{R}$. 

On the other hand, we deduce that
the right hand side of (\ref{HJB}) is maximized at
\begin{equation*}
\boldsymbol{\pi }^{\ast }(u;x)=\mathcal{N}\left( u\, \Big | \  \
\frac{CDxv^{\prime \prime }(x)+Bv^{\prime }(x)-Rx-Q}{N-D^{2}v^{\prime \prime
}(x)}\ ,\  \frac{\lambda }{N-D^{2}v^{\prime \prime }(x)}\right).
\end{equation*}%
Let $\pi ^{\ast }=\{\pi ^{\ast }_t,t\geq0\}$ be the open-loop control distribution
generated from the above feedback law along with the corresponding state process
$\{X_{t}^{\ast },t\geq0\}$  with $X_{0}^{\ast }=x$, and assume for now that $\pi ^{\ast }\in \mathcal{A}(x)$.
Then
\begin{equation*}
\mathbb{E}\left[ e^{-\rho T}v(X_{T}^{\ast })\right] =v(x)-\mathbb{E}\left[
\int_{0}^{T}e^{-\rho t}\left( \tilde{r}(X_{t}^{\ast },\pi _{t}^{\ast
})-\lambda \int_{\mathbb{R}}\pi _{t}^{\ast }(u)\ln \pi _{t}^{\ast
}(u)du\right) dt\right] .
\end{equation*}%
Noting that
$\liminf_{T\rightarrow \infty }\mathbb{E}\left[
e^{-\rho T}v(X_{T}^{\ast })\right] \leq \limsup_{T\rightarrow \infty }\mathbb{E}\left[
e^{-\rho T}v(X_{T}^{\ast })\right]=0$,
and applying the dominated
convergence theorem yield
\begin{equation*}
v(x)\leq \mathbb{E}\left[ \int_{0}^{\infty }e^{-\rho t}\left( \tilde{r}%
(X_{t}^{\ast },\pi _{t}^{\ast })-\lambda \int_{\mathbb{R}}\pi _{t}^{\ast
}(u)\ln \pi _{t}^{\ast }(u)du\right) dt\right] ,
\end{equation*}%
for any $x\in \mathbb{R}$. This proves that $v$ is indeed the value function, namely $v\equiv V$.

It remains to show that $\pi ^{\ast }\in \mathcal{A}(x)$. First, we verify
that
\begin{equation}
\liminf_{T\rightarrow \infty }e^{-\rho T}\mathbb{E}\left[ (X_{T}^{\ast })^{2}%
\right] =0,  \label{desired}
\end{equation}%
where $\{X^{\ast }_t,t\geq 0\}$ solves the SDE (\ref{second_state_dynamics}). To this end, It\^{o}'s formula yields, for any $T\geq 0,$
$$(X_{T}^{\ast })^{2}=x^{2}+\int_{0}^{T}\left( 2\left(\tilde{A}X_{t}^{\ast }+\tilde{B}%
\right) X_{t}^{\ast }+(\tilde{C_1}X_{t}^{\ast }+\tilde{C_2})^2+D^2\right) \,dt$$
\begin{equation}
+\int_{0}^{T}2X_{t}^{\ast }%
\sqrt{\left(\tilde{C_1}X_{t}^{\ast }+\tilde{C_2}\right)^{2}+\tilde{D}^2}\ dW_{t}.
\label{Ito_X*_squared}
\end{equation}%
Following similar arguments as in the proof of Lemma \ref{Lemma_appendix} in Appendix C, we can show that $\mathbb{E}[(X^*_T)^2]$ contains the terms $e^{(2\tilde{A}+\tilde{C_1}^2)T}$ and $e^{\tilde{A}T}$.

If $2\tilde{A}+\tilde{C_1}^2\leq \tilde{A}$, then  $\tilde{A}\leq 0$, in which case (\ref{desired}) easily follows. Therefore, to show (\ref{desired}), it remains to consider the case in which the term  $e^{(2\tilde{A}+\tilde{C_1}^2)T}$ dominates $e^{\tilde{A}T}$, as $T\rightarrow\infty$. In turn, using that $k_2$ solves the equation (\ref{anzats_a_2}), we obtain
$$2\tilde{A}+\tilde{C_1}^2-\rho=2A+\frac{2B(k_2(B+CD)-R)}{N-k_2D^2}+\left(C+\frac{D(k_2(B+CD)-R)}{N-k_2D^2}\right)^2-\rho$$
$$=2A+C^2-\rho+\frac{2(B+CD)(k_2(B+CD)-R)}{N-k_2D^2}+\frac{D^2(k_2(B+CD)-R)^2}{(N-k_2D^2)^2}$$
\begin{equation}\label{decay_condition}
=2A+C^2-\rho+\frac{k_2(2N-k_2D^2)(B+CD)^2}{N-k_2D^2}-\frac{2NR(B+CD)-D^2R^2}{N-k_2D^2}.
\end{equation}
Notice that the first fraction is nonpositive due to $k_2<0$, while the second fraction is bounded for any $k_2<0$. Using Assumption \ref{Assumption} on the range of $\rho$, we then easily deduce (\ref{desired}).

Next, we establish the admissibility constraint
$$%
\mathbb{E}\left[ \int_{0}^{\infty }e^{-\rho t}\left \vert L(X_{t}^{\ast
},\pi _{t}^{\ast })\right \vert dt\right] <\infty. $$
The definition of
$L$ and the form of $r(x,u)$ yield
\begin{equation*}
\mathbb{E}\left[ \int_{0}^{\infty }e^{-\rho t}\left \vert L(X_{t}^{\ast
},\pi _{t}^{\ast })\right \vert \, dt\right]
\end{equation*}%
\begin{equation*}
=\mathbb{E}\left[ \int_{0}^{\infty }e^{-\rho t}\left \vert \int_{\mathbb{R}%
}r(X_{t}^{\ast },u)\pi _{t}^{\ast }(u)du-\lambda \int_{\mathbb{R}}\pi^{\ast}
_{t}(u)\ln \pi^{\ast} _{t}(u)du\right \vert \,dt\right]
\end{equation*}%
\begin{equation*}
=\mathbb{E}\Big[ \int_{0}^{\infty }e^{-\rho t}\Big| \int_{\mathbb{R}}-\left(%
\frac{M}{2} \left( X_{t}^{\ast }\right) ^{2}+RX_{t}^{\ast }u+\frac{N}{2}u^2+PX_{t}^{\ast }+Qu\right) \pi
_{t}^{\ast }(u)du
\end{equation*}
\begin{equation*}
+\frac{\lambda }{2}\ln \left( \frac{2\pi e\lambda }{N-k _{2}D^{2}}\right) \Big| \, dt\Big],
\end{equation*}%
where we have applied similar computations as in the proof of Theorem \ref{exploration_cost_theorem}. Recall that
\[
\pi^{\ast }_t(u)=\mathcal{N}\left( u\  \left \vert \frac{(k_{2}(B+CD)-R)X^{\ast }_t+k_1B-Q}{N-k_{2}D^{2}}\ ,\  \frac{\lambda }{N-k_{2}D^{2}}\right. \right) ,\;\;t\geq 0.
\]
It is then clear that it suffices to prove $\mathbb{E}\left[ \int_{0}^{\infty }e^{-\rho
t}(X_{t}^{\ast })^{2}dt\right] <\infty ,$ which follows easily since, as
shown in (\ref{decay_condition}), $\rho >$ $2\tilde{A}+\tilde{C_1}^2$ under Assumption \ref{Assumption}. The remaining admissibility
conditions for $\pi ^{\ast }$ can be easily verified.

\section*{Appendix C: Proof of Theorem 7}
We first note that when (a) holds, the function $v$ solves the HJB equation (\ref{HJB_second_case}) of the exploratory LQ problem. Similarly for the classical LQ problem when (b) holds.

Next, we prove the equivalence between (a) and (b).
First, a comparison between the two
HJB equations (\ref{HJB_second_case}) and  (\ref{classical_HJB})
yields that if $v$ in (a) solves the former, then $w$ in (b) solves the
latter, and vice versa.

Throughout this proof, we let $x$ be fixed, being the initial state of both the 
exploratory problem in statement (a) and the classical problem in statement (b). 
Let $\pi^*=\{\pi^*_t,t\geq0\}$ and $u^*=\{u^*_t,t\geq0\}$ be respectively the open-loop controls generated by the feedback controls $\boldsymbol{\pi}^{\ast }$ and $\boldsymbol{u}^{\ast }$ of the two problems, and $X^*=\{X^*_t,t\geq0\}$ and $x^*=\{x^*_t,t\geq0\}$ be respectively the corresponding state processes, both starting from $x$. 
It remains to show the equivalence between the admissibility of $\pi^*$ for the exploratory problem and that of $u^*$ for the classical problem.
%
%
To this end, we first compute ${\mathbb{E}}[(X_{T}^{\ast })^{2}]$ and ${\mathbb{E}}[(x_{T}^{\ast })^{2}]$.

To ease the presentation, we rewrite the exploratory dynamics of $X^{\ast }$
 under $\pi ^{\ast }$ as
\begin{equation*}
dX_{t}^{\ast }=\left( AX_{t}^{\ast }+B\frac{(\alpha _{2}(B+CD)-R)X_{t}^{\ast
}+\alpha _{1}B-Q}{N-\alpha _{2}D^{2}}\right) \ dt
\end{equation*}%
\begin{equation*}
+\sqrt{\left( CX_{t}^{\ast }+D\frac{(\alpha _{2}(B+CD)-R)X_{t}^{\ast }+\alpha
_{1}B-Q}{N-\alpha _{2}D^{2}}\right) ^{2}+\frac{\lambda D^{2}}{N-\alpha
_{2}D^{2}}}\ dW_{t}
\end{equation*}%
\begin{equation*}
=(A_{1}X_{t}^{\ast }+A_{2})\ dt+\sqrt{\left( B_{1}X_{t}^{\ast
}+B_{2}\right) ^{2}+C_{1}}\ dW_{t},
\end{equation*}%
where $A_{1}:=A+\frac{B(\alpha _{2}(B+CD)-R)}{N-\alpha _{2}D^{2}}$, $A_{2}:=%
\frac{B(\alpha _{1}B-Q)}{N-\alpha _{2}D^{2}}$, $B_{1}:=C+\frac{D(\alpha
_{2}(B+CD)-R)}{N-\alpha _{2}D^{2}}$,
\begin{flushleft}
$B_{2}:=\frac{D(\alpha _{1}B-Q)}{%
N-\alpha _{2}D^{2}}$ and $C_{1}:=\frac{\lambda D^{2}}{N-\alpha _{2}D^{2}}$.
\end{flushleft}

Similarly, the classical dynamics of $x^{\ast }$
 under  $u^{\ast }$  solves
\begin{equation*}
dx_{t}^{\ast }=(A_{1}x_{t}^{\ast }+A_{2})\ dt+( B_{1}x_{t}^{\ast
}+B_{2}) \ dW_{t}.
\end{equation*}%

The desired equivalence of the admissibility then follows from the following lemma.

\begin{lemma}\label{Lemma_appendix}
We have that (i) $\liminf_{T\rightarrow \infty}e^{-\rho T}\mathbb{E}\big[\big(X_T^{*}\big)^2\big]=0$ if and only if
$\liminf_{T\rightarrow \infty}e^{-\rho T}{\mathbb{E}}\big[\big({x}_T^{*}\big)^2\big]=0$; (ii) $\mathbb{E}\left[\int_0^\infty e^{-\rho t}\big(X_t^{*}\big)^2dt\right]<\infty$ if and only if ${\mathbb{E}}\left[\int_0^\infty e^{-\rho t}\big({x}_t^{*}\big)^2dt\right]<\infty$.
\end{lemma}
\begin{proof}
%
Denote $n(t):=\mathbb{E}\left[X_{t}^{*}\right]$, for $t\geq 0$. Then,
a standard argument involving a series of stopping times and the dominated convergence theorem yields
the ODE
$$\frac{dn(t)}{dt}=A_1n(t)+A_2,\quad n(0)=x,$$
whose solution is $n(t)=\left(x+\frac{A_2}{A_1}\right)e^{A_1t}-\frac{A_2}{A_1}$, if $A_1\neq 0$, and $n(t)=x+A_2t$, if $A_1=0$. Similarly, the function $m(t):=\mathbb{E}\big[(X_t^{*})^2\big]$, $t\geq 0$, solves the ODE
$$\frac{dm(t)}{dt}=(2A_1+B_1^2)m(t)+2(A_2+B_1B_2)n(t)+B_2^2+C_1,\quad m(0)=x^2.$$

We can also show that $n(t)={\mathbb{E}}\big[{x}_t^{*}\big]$, and deduce that $\hat{m}(t):={\mathbb{E}}\big[({x}_t^{*})^2\big]$, $t\geq 0$, satisfies
$$\frac{d\hat{m}(t)}{dt}=(2A_1+B_1^2)\hat{m}(t)+2(A_2+B_1B_2)n(t)+B_2^2,\quad \hat{m}(0)=x^2.$$
Next, we find explicit solutions to the above ODEs corresponding to various conditions on the parameters.

(a) If $A_1=B_1^2=0$, then direct computation gives $n(t)=x+A_2t$, and
$$m(t)=x^2+A_2(x+A_2t)t+(B_2^2+C_1)t,$$
$$\hat{m}(t)=x^2+A_2(x+A_2t)t+B_2^2t.$$

(b) If $A_1=0$ and $B_1^2\neq 0$, we have $n(t)=x+A_2t$, and
$$m(t)=\left(x^2+\frac{2(A_2+B_1B_2)\left(A_2+B_1^2(x+B_2^2+C_1)\right)}{B_1^4}\right)e^{B_1^2t}$$
$$-\frac{2(A_2+B_1B_2)\left(A_2+B_1^2(x+B_2^2+C_1)\right)}{B_1^4},$$
$$\hat{m}(t)=\left(x^2+\frac{2(A_2+B_1B_2)\left(A_2+B_1^2(x+B_2^2)\right)}{B_1^4}\right)e^{B_1^2t}$$
$$-\frac{2(A_2+B_1B_2)\left(A_2+B_1^2(x+B_2^2)\right)}{B_1^4}.$$

(c) If $A_1\neq 0$ and $A_1+B_1^2=0$, then $n(t)=\left(x+\frac{A_2}{A_1}\right)e^{A_1t}-\frac{A_2}{A_1}$. Further calculations yield
$$m(t)=\left(x^2+\frac{A_1(B_2^2+C_1)-2A_2(A_2+B_1B_2)}{A_1^2}\right)e^{A_1t}$$
$$+\frac{2(A_2+B_1B_2)(A_1x+A_2)}{A_1}te^{A_1t}-\frac{A_1(B_2^2+C_1)-2A_2(A_2+B_1B_2)}{A_1^2},$$
$$\hat{m}(t)=\left(x^2+\frac{A_1B_2^2-2A_2(A_2+B_1B_2)}{A_1^2}\right)e^{A_1t}$$
$$+\frac{2(A_2+B_1B_2)(A_1x+A_2)}{A_1}te^{A_1t}-\frac{A_1B_2^2-2A_2(A_2+B_1B_2)}{A_1^2}.$$

(d) If $A_1\neq 0$ and $2A_1+B_1^2=0$, we have $n(t)=\left(x+\frac{A_2}{A_1}\right)e^{A_1t}-\frac{A_2}{A_1}$, and
$$m(t)=\frac{2(A_2+B_1B_2)(A_1x+A_2)}{A_1^2}e^{A_1t}$$
$$+\frac{A_1(B_2^2+C_1)-2A_2(A_2+B_1B_2)}{A_1^2}t+x^2-\frac{2(A_2+B_1B_2)(A_1x+A_2)}{A_1^2},$$
$$\hat{m}(t)=\frac{2(A_2+B_1B_2)(A_1x+A_2)}{A_1^2}e^{A_1t}$$
$$+\frac{A_1B_2^2-2A_2(A_2+B_1B_2)}{A_1^2}t+x^2-\frac{2(A_2+B_1B_2)(A_1x+A_2)}{A_1^2}.$$

(e) If $A_1\neq 0$, $A_1+B_1^2\neq 0$ and $2A_1+B_1^2\neq 0$, then we arrive at $n(t)=\left(x+\frac{A_2}{A_1}\right)e^{A_1t}-\frac{A_2}{A_1}$, and
$$m(t)=$$
$$\left(x^2+\frac{2(A_2+B_1B_2)(A_1x+A_2)}{A_1(A_1+B_1^2)}+\frac{A_1(B_2^2+C_1)-2A_2(A_2+B_1B_2)}{A_1(2A_1+B_1^2)}\right)e^{(2A_1+B_1^2)t}$$
$$-\frac{2(A_2+B_1B_2)(A_1x+A_2)}{A_1(A_1+B_1^2)}e^{A_1t}-\frac{A_1(B_2^2+C_1)-2A_2(A_2+B_1B_2)}{A_1(2A_1+B_1^2)},$$
$$\hat{m}(t)=\left(x^2+\frac{2(A_2+B_1B_2)(A_1x+A_2)}{A_1(A_1+B_1^2)}+\frac{A_1B_2^2-2A_2(A_2+B_1B_2)}{A_1(2A_1+B_1^2)}\right)e^{(2A_1+B_1^2)t}$$
$$-\frac{2(A_2+B_1B_2)(A_1x+A_2)}{A_1(A_1+B_1^2)}e^{A_1t}-\frac{A_1B_2^2-2A_2(A_2+B_1B_2)}{A_1(2A_1+B_1^2)}.$$
It is easy to see that for all cases (a)--(e), the assertions in the Lemma follow and we conclude.
\end{proof}
\\

\end{appendix}

\bibliography{mybibtex}

\end{document}